\documentclass[11pt]{article}

\usepackage{amsmath,amsthm,amscd,amssymb}

\usepackage{latexsym}

\usepackage{mathrsfs}

\newcommand{\bbN}{{\mathbb{N}}}

\newcommand{\bbR}{{\mathbb{R}}}

\newcommand{\bbH}{{\mathbb{H}}}
\newcommand{\bbE}{{\mathbb{E}}}

\newcommand{\bbZ}{{\mathbb{Z}}}

\newcommand{\bbC}{{\mathbb{C}}}

\newcommand{\calA}{\mathcal A}

\newcommand{\calB}{\mathcal B}

\newcommand{\calC}{\mathcal C}

\newcommand{\no}{\nonumber}

\newcommand{\eps}{\varepsilon}

\newcommand{\beq}{\begin{equation}}

\newcommand{\eeq}{\end{equation}}

\newcommand{\ba}{\begin{align}}

\newcommand{\ea}{\end{align}}

\renewcommand{\Re}{\mathop{\mathrm{Re}}}
\renewcommand{\Im}{\mathop{\mathrm{Im}}}

\DeclareMathOperator{\Tr}{Tr}
\DeclareMathOperator{\Var}{Var}

\allowdisplaybreaks

\numberwithin{equation}{section}

\newtheorem{theorem}{Theorem}[section]

\newtheorem{proposition}[theorem]{Proposition}

\newtheorem{lemma}[theorem]{Lemma}

\newtheorem{corollary}[theorem]{Corollary}

\theoremstyle{definition}

\theoremstyle{remark}

\newtheorem{remark}{Remark}[section]
\newtheorem{hypothesis}{Hypothesis}[section]

\title{Universality of mesoscopic fluctuations for orthogonal polynomial ensembles}
\author{Jonathan Breuer \footnote{Einstein Institute of Mathematics, The Hebrew University of Jerusalem, Jerusalem
91904, Israel. E-mail: jbreuer@math.huji.ac.il. Supported in part by the US-Israel Binational Science Foundation (BSF) Grant no.\ 2010348 and by the Israel Science Foundation (ISF) Grant no.\ 1105/10.}  \and Maurice Duits \footnote{Department of Mathematics, Stockholm University,  SE-106 91 Stockholm, Sweden. Email: duits@math.su.se. Supported in part by the grant KAW 2010.0063 from the Knut and Alice Wallenberg Foundation and by the Swedish Research Council (VR) Grant no.\ 2012-3128.} }
\date{\today}

\begin{document}
\maketitle
\begin{abstract}
We prove that the fluctuations of mesocopic linear statistics for orthogonal polynomial ensembles are universal in the sense that two measures with asymptotic recurrence coefficients have the same asymptotic mesoscopic fluctuations (under an additional assumption on the local regularity of one of the measures). The convergence rate of the recurrence coefficients determines the range of scales on which the limiting fluctuations are identical. Our main tool is an analysis of the Green's function for the associated Jacobi matrices. As a  particular consequence we obtain a Central Limit Theorem  for the modified Jacobi Unitary Ensembles on all mesosopic scales.
\end{abstract}

\section{Introduction and statement of results}

Let $\mu$ be a Borel measure on $\bbR$ with finite moments (i.e., $\int_{\bbR} |x|^k {\rm d} \mu(x) <\infty$ for $k=0,1,2,\ldots$). The orthogonal polynomial ensemble of size $n \in \bbN$  is defined as the probability measure on $\bbR^n$ proportional to 
\begin{equation}\label{eq:defOPE}
\prod_{1\leq i<j  \leq n } (x_i-x_j)^2 {\rm d} \mu(x_1) \cdots {\rm d} \mu(x_n).
\end{equation}
Many interesting models in probability, particularly in random matrix theory, lead to orthogonal polynomial ensembles, (for a survey of such models we refer to \cite{K}). A relevant example here are the modified Jacobi Unitary Ensembles. These are the orthogonal polynomial ensembles with   
\begin{equation}\label{eq:JacobiWeight}
{\rm d} \mu(x)=h(x) (1-x)^{\gamma_1}(1+x)^{\gamma_2}{\rm d} x, \qquad \text{ on } [-1,1],
\end{equation}
with $\gamma_1,\gamma_2 >-1$ and $h$ a function that is analytic in a neighborhood of $[-1,1]$ and strictly positive on $[-1,1]$. When $h=1$ this weight is the classical Jacobi weight. 

A significant amount of research on orthogonal polynomial ensembles (and other random matrix models) has focused on proving universality results, with a large part focusing on what may be called \emph{microscopic} universality. To give an example, in the case of the Jacobi weight \eqref{eq:JacobiWeight} the random points $x_j$ accumulate on the interval $[-1,1]$ as $n \to \infty$, and, with probability one, their configuration (that is, the empirical measure) converges to the arcsine law on $[-1,1]$, see \cite{KMVV}. This means that when we zoom in around a point in the bulk $x_0 \in (-1,1)$, the distance between neighboring points is of order $\sim 1/n$. On this scale, the local correlations between these neighboring points near $x_0$ are governed by the \emph{sine} process \cite{KV}. The universality conjecture says that this is not special for the Jacobi weight. Indeed, the sine process appears as a limit in the bulk for a wide class of measure $\mu$. In the past two decades, this has been proved in many different settings but we do not review this here. 

In this paper, we will be interested in a different type of universality, namely that of the \emph{mesoscopic fluctuations} in the bulk.
We shall prove a Central Limit Theorem (i.e., asymptotic normality of these fluctuations) for a large class of measures supported on an interval.  

To be more explicit, let $x_0 \in (-1,1)$ and $0<\alpha<1$. For a function $f:\bbR \to \bbR$ we define the linear statistic $X^{(n)}_{f,\alpha,x_0}$ by 
\begin{equation}
X^{(n)}_{f,\alpha,x_0}= \sum_{j=1}^n f(n^\alpha(x_j-x_0)).
\end{equation}
Note that if $f$ has compact support, the linear statistic will only depend on particles that are at distance $\sim n^{-\alpha
}$ to $x_0$. Since $0<\alpha<1$ we refer to such distances as the \emph{mesoscopic} scales, compared to the microscopic and macroscopic scales that would correspond to $\alpha=1$ and $\alpha=0$ respectively. 

A consequence of our main result is the following result for the Jacobi weight: 

\begin{theorem} \label{thm:JacobiCLT}
Consider the orthogonal polynomial ensemble with the modified Jacobi weight \eqref{eq:JacobiWeight} with $\gamma_1,\gamma_2 >-1$ and  $h$ a function that is analytic in a neighborhood of $[-1,1]$ and strictly positive on $[-1,1]$.
For any $f \in C_c^1(\bbR)$ $($the continuously differentiable functions with compact support$)$ and for any $0<\alpha<1$ we have that 
\begin{equation}\label{eq:CLTJacobi}
X^{(n)}_{f,\alpha,x_0}-\mathbb E X^{(n)}_{f,\alpha,x_0} \to N(0, \sigma_f^2),
\end{equation}
where 
\begin{equation}\label{eq:CLTJacobivariance} 
\sigma_f^2= \frac{1}{4 \pi^2} \iint \left( \frac{f(x)-f(y)}{x-y} \right)^2 {\rm d} x {\rm d} y.
\end{equation}
\end{theorem}

The conclusion of the theorem  is striking for several reasons. First of all, in contrast to the independent variable case where a normalization of $\sqrt{n}$ is needed, Central Limit Theorems for these correlated systems generally hold without any normalization. This is due to the repulsion between the random points \cite{Diac}.  Note also that the limiting variance is scale invariant, in the sense that if $g(x)=f(px)$ then $\sigma_f^2=\sigma_g^2$. This seems to be connected to the independence of the variance from $\alpha$ as well. Finally,  the limiting variance in  \eqref{eq:CLTJacobivariance} does not depend on $\gamma_1$, $\gamma_2$ and $h$, making the theorem universal with respect to these parameters of the orthogonal polynomial ensemble. In fact,  it is believed (see for example \cite{P}) that \eqref{eq:CLTJacobi} holds for a large class of measures $\mu$ as long as $x_0$ is in the \emph{bulk}.  In the context of random matrix models similar results have been proved for the CUE and other classical compact groups  \cite{S}, GOE \cite{BdmV1}, symmetric Wigner matrices \cite{BdmV2}, GUE \cite{FKS}, in the context of Brownian motion \cite{DJ} and for Gaussian $\beta$-ensembles \cite{B}.

The purpose of the present paper is to prove \eqref{eq:CLTJacobi} and \eqref{eq:CLTJacobivariance} for a  wider class of measures, validating universality of the mesoscopic fluctuations for this class. We shall in fact show that whenever two measures are sufficiently `close' in a certain sense, their associated orthogonal polynomial ensembles have the same asymptotic fluctuations down to a corresponding mesoscopic scale. The approach that we follow is inspired by our earlier work \cite{BD-CLT} where we use the recurrence coefficients associated with the orthogonal polynomials associated with $\mu$ to understand the asymptotics of the \emph{macroscopic} fluctuations of the orthogonal polynomial ensemble.

Given a measure, $\mu$, we denote the ortho\emph{normal} polynomials with respect to $\mu$ by $\{p_n(x)\}_{n=0}^\infty$. Here each $p_n$ is a polynomial of degree $n$ and $$\int_\bbR p_n(x) p_m(x) {\rm d }\mu(x)=\delta_{nm},\qquad n,m=0,1,2\ldots$$
These polynomials famously satisfy a three term recurrence relation
$$x p_n(x)=a_{n+1} p_{n+1}(x)+b_{n+1} p_n(x) +a_{n} p_{n-1}(x),$$
with $a_n \geq 0$, $b_n \in \bbR$, and we assume that $a_0=0$. 
For later reference we introduce here the associated Jacobi matrix as well. This is simply the tridiagonal matrix
\beq \no
J_\mu=
\begin{pmatrix}
b_1 & a_1 & 0 & \dots \\
a_1 & b_2 & a_2 & \dots \\
0 & a_2 & b_3 & \dots \\
\vdots & \vdots & \vdots & \ddots
\end{pmatrix}
\eeq
which we view as an operator on $\ell^2(\bbN)$. For the intimate connection between the spectral theory of Jacobi matrices and the theory of orthogonal polynomials see, e.g., \cite{deift}.

Our first main result says that if the recurrence coefficients of $\mu$ and $\mu_0$ are asymptotically sufficiently close, then their respective orthogonal polynomial ensembles have the same fluctuations around a given point $x_0$ at a corresponding mesoscopic scale. More precisely, the difference of the moments  $X^{(n)}_{f,\alpha,x_0}$  with respect to the two different orthogonal polynomial ensembles tends to zero as $n\to \infty$, and $\alpha$ depends on the rate of approach of the recurrence coefficients to each other.

\begin{theorem} \label{thm:main-result}
Let $\mu$ and $\mu_0$ be two measures with finite moments and denote by $\{a_n, b_n\}_{n=1}^\infty$ and $\{a_n^0,b_n^0\}_{n=1}^\infty$ the respective associated recurrence coefficients. Let $x_0 \in \bbR$ be such that there exists a neighborhood $x_0 \in I$ on which the following two conditions are satisfied: \\
(i) $\mu_0$ restricted to $I$ is absolutely continuous with respect to Lebesgue measure and its Radon-Nikodym derivative is bounded there.  \\
(ii) The orthonormal polynomials for $\mu_0$ are uniformly bounded on $I$.\\
Assume further that
\beq \label{eq:decay-rate-condition}
a_n-a_n^0=\mathcal O(n^{-\beta}), \qquad b_n-b_n^0=\mathcal O(n^{-\beta}),
\eeq
as $n \to \infty$ for some $1>\beta>0$.

Then for any $f\in C_c^1(\bbR)$ and any $0<\alpha<\beta$ we have
\beq \label{eq:conclusion}
\left|\bbE \left(X^{(n)}_{f,\alpha,x_0}-\bbE X^{(n)}_{f,\alpha,x_0} \right)^m -\bbE_0 \left(X^{(n)}_{f,\alpha,x_0}-\bbE_0 X^{(n)}_{f,\alpha,x_0}\right)^m\right| \rightarrow 0
\eeq
as $n \rightarrow \infty$. Here $\bbE$ and $\bbE_0$ denote the expection with respect the orthogonal polynomial ensemble corresponding to $\mu$ and $\mu_0$ respectively.
\end{theorem}

Note that this theorem is a true `universality' result: without specifying the limiting distribution (if it exists at all), the conclusion of the theorem is that the moments of the fluctuation are the same for all OPE's that share the same asymptotic behavior for the associated recurrence coefficients. 

A few remarks are in order:

\begin{remark}
A different way of thinking about Theorem \ref{thm:main-result} is as a stability result. It gives the conditions under which a perturbation of a Jacobi matrix does not change the asymptotics of the associated orthogonal polynomial ensemble on a particular scale. Stability of universality in the microscopic scale under perturbations of the recurrence coefficients has been recently studied in \cite{BLS}. We note that the conditions given here on the perturbation are restricted to its rate of decay. It is an interesting problem to consider if one may get a stronger conclusion by imposing more regularity on the perturbation (such as monotonicity).
\end{remark}

\begin{remark}
Condition (ii) above is slightly stronger than the one we actually need, which is a boundedness condition on certain Cesaro averages of polynomials on shrinking neighborhoods of $x_0$ (see Proposition \ref{Jacobi-perturbation} and in particular \eqref{eq:measure-condition}). However, in many sufficiently nice cases, this stronger condition is known to hold. In fact, for the modified Jacobi Unitary Ensembles \eqref{eq:JacobiWeight} one can even compute precise asympotics for the orthogonal polynomials using Riemann-Hilbert techniques, as was done in \cite[Theorems 1.6 and 1.12]{KMVV}, which prove the boundedness of the polynomials. 

The phrasing in condition (ii) also emphasizes the interesting connection to  the natural problem of bounding orthogonal  polynomials in a general context. The connection between the boundedness of orthogonal polynomials and absolute continuity of $\mu$ has been quite extensively studied in both the orthogonal polynomial and Schr\"odinger operator community. In fact, boundedness of both the orthogonal and second-kind orthogonal polynomials implies absolute continuity of the measure $\mu$ and the bound on the polynomials even gives a bound on the Radon-Nikodym derivative of the measure. This has been known for some time now \cite{JL, Pearson, Simon-Boundedness}. The reverse direction, however, was recently shown to be false (see \cite{rakhmanov} for discussions of the related Steklov conjecture, and \cite{Avila} for a counterexample to the associated Schr\"odinger conjecture). 
\end{remark}

\begin{remark} \label{rem:properties-of-mu}
Finally, although this theorem refers to the mesoscopic scales, and thus is a `local' result, we note that $a_n-a_n^0 \rightarrow 0$ and $b_n-b_n^0 \rightarrow 0$ imply that the supports of $\mu_0$ and $\mu$ coincide up to at most a countable discrete set (i.e., their \emph{essential} supports coincide). This follows immediately from Weyl's theorem on the essential spectrum of self-adjoint operators \cite{reed-simon1} applied to $J_0$ and $J$. Furthermore, although no apriori information is given on the regularity of $\mu$, for sufficiently large $\beta$ some regularity may be deduced. In particular, if $\beta>1$, $\mu$ will be purely absolutely continuous precisely where $\mu_0$ is \cite{RS3}. If $\beta>1/2$ and $\{a_n^0, b_n^0\}_{n}$ are constant or periodic then it is known that one may define a set where the Radon-Nikodym derivatives of both $\mu$ and $\mu_0$ with respect to Lebesgue measure are simultaneously finite and positive and both vanish outside this set \cite{dk,killip} (this is conjectured to be true for general bounded $a_n^0$ and $b_n^0$ \cite{kls}). For $\beta \leq 1/2$ there is no guarantee that $\mu$ has any absolutely continuous component, and in fact $\mu$ may even be a pure point measure.  
\end{remark}

A particular consequence of Theorem \ref{thm:main-result} is  that if we can prove a mesoscopic CLT for a given special measure $\mu_0$, then we have a mesocopic CLT for any orthogonal polynomial ensemble that can be compared to that special case. We will prove such a CLT for the case of ${\rm d}\mu_0(x)=\frac{\sqrt{4-x^2}}{2\pi} \chi_{[-2,2]}(x){\rm d}x$ which corresponds to $a_n^0=1$ and $b_n^0=0$. This case, for which the orthogonal polynomials are the (rescaled) Chebyshev polynomials of the second kind, is known as the `free case' because of the fact that the associated Jacobi matrix is the free Laplace operator on $\ell^2(\bbN)$. In this case, the resolvent $\left(J_0-z\right)^{-1}$ is an exponentially decaying perturbation of a Toeptliz matrix. We shall use the connection of Toeplitz and Fredholm determinants in proving the CLT directly for this free case. Combined with Theorem \ref{thm:main-result}, this then leads to the following theorem:

\begin{theorem}\label{thm:CLTgeneral}
Suppose there exists $a>0,b\in \bbR$ and $0<\beta<1$  such that 
$$a_n=a+\mathcal O(n^{-\beta}), \qquad b_n=b+\mathcal O(n^{-\beta}),$$
as $n\to \infty$.   Then for any  $f\in C_c^1(\bbR)$, $x_0\in (b-2a,b+2a) $  and $0<\alpha<\beta$ we have that 
\begin{equation}\label{eq:CLTgeneral}
X^{(n)}_{f,\alpha,x_0}-\mathbb E X^{(n)}_{f,\alpha,x_0} \to N(0, \sigma_f^2),
\end{equation}
where 
\begin{equation}\label{eq:CLTgeneralvariance} 
\sigma_f^2= \frac{1}{4 \pi^2} \iint \left( \frac{f(x)-f(y)}{x-y} \right)^2 {\rm d} x {\rm d} y.
\end{equation}
\end{theorem} 
Theorem \ref{thm:JacobiCLT} follows from Theorem \ref{thm:CLTgeneral} and the results in \cite{KMVV}. Indeed, in the latter paper the authors compute various precise asymptotics for the orthogonal polynomials for the modified Jacobi Unitary Ensembles and their features, based on the Riemann-Hilbert approach, and  in particular, \cite[Theorem 1.10]{KMVV} 
$$a_n=1/2+ \mathcal O(n^{-2}), \text{  and  } b_n= \mathcal O(n^{-2}),$$
as $n \to \infty$.

As remarked above (see Remark \ref{rem:properties-of-mu}), the condition that the recurrence coefficients have a limit implies that the essential support of $\mu$ consists of a single interval, and if $\beta>1$ we may also deduce that $\mu$ is purely absolutely continuous there. This however does not imply that \emph{any} purely absolutely continuous measure $\mu$ supported on an interval satisfies the conclusion of Theorem \ref{thm:CLTgeneral}. Although it is known \cite[Theorem 1.4.2]{SimonSzego} that the recurrence coefficients for such a measure will have a limit, our analysis requires the knowledge of the rate of approach as well. Whether our condition on the decay rate may be weakened by combining them with regularity properties of $\mu$ is a problem we currently leave open.

It is known that the one-interval assumption on the support of the measure $\mu$ is necessary for a macroscopic Central Limit Theorem (i.e. $\alpha=0$) \cite{Shch}. However, on the mesoscopic scale, the number of intervals in the support is believed to be irrelevant. Hence it should be possible to extend Theorem \ref{thm:CLTgeneral} to the case of multi-interval case, for which it is sufficient by Theorem \ref{thm:main-result} to  prove a Central Limit Theorem for a particular multi-cut case. To the best of our knowledge, no such example exists and we leave this as an open problem.

As remarked above, our approach is inspired by our work in \cite{BD-CLT} where we deduce a macroscopic CLT for polynomial $f$ by studying the asymptotics of the recurrence coefficients associated with $\mu$. The main idea is that each cumulant of $X_f^{(n)}$ may be expressed as the trace of a linear combination of powers of $f\left(J_\mu \right)$ multiplied by $P_n$, where $P_n$ is the projection onto the first $n$ coordinates in $\ell^2$ (see \eqref{eq:defPn} below). What makes the analysis go through in that case is the fact that for polynomial $f$, $f(J)$ is a banded matrix, which then implies that only a finite number of recurrence coefficients determine each cumulant. 

The issue with using this approach in the mesoscopic scale is the fact that polynomial $f$ are inappropriate for studying this scale. Rather, to study the mesoscopic scale one needs an $f$ that decays at infinity. The problem is that for such $f$, $f(J)$ will no longer be banded. However, for $f(x)=\frac{1}{x-(x_0+\eta)}$ with $\Im \eta>0$ the Combes-Thomas estimate (see Proposition \ref{prop:CT} below) says that the entries of $f(J)$ decay exponentially with their distance from the main diagonal. Thus, $f(J)$ is `effectively' banded, which leads to the hope that with some technical effort it should be possible to apply the strategy described above to deduce a CLT for $f$. That this is indeed the case is demonstrated and exploited in Section \ref{sec:comparison}.

The rest of the paper is structured as follows. In Section \ref{sec:pre} we discuss the cumulant approach and mention some results of our previous works \cite{BD-Nevai,BD-CLT} that we will need here. Some preliminaries from operator theory, including the Combes-Thomas estimate, are described in Section \ref{sec:pre} as well. In Section \ref{sec:comparison} we prove Theorem \ref{thm:main-result} in the case $f(x)= \sum_{j=1}^N c_j \Im (x-\eta_j)^{-1}$. For that same class of functions Section \ref{sec:CLTfree} has both a proof of a Central Limit Theorem for the free case (i.e., $a_n=1$ and $b_n=0$), and of Theorem \ref{thm:CLTgeneral}. Finally, in Section \ref{sec:extension} we extend our  results so that they hold for $f\in C^1_c(\bbR)$.

%%%%%%%%%%%%%%%%%%%%%%%%%%%%%%%%%%%%%%%%%%%%Preliminaries%%%%%%%

\section{Preliminary considerations on cumulants}\label{sec:pre}

In this section we fix some notation and recall some basic facts regarding cumulants for linear statistics. We also review some results from our earlier works \cite{BD-Nevai,BD-CLT} that will be needed later on. Finally, we formulate the classical Combes-Thomas estimate \cite{CT} that is central to everything that follows. 

\subsection{Cumulants}
We recall that for  a fixed measure, $\mu$, we consider the corresponding orthogonal polynomial ensemble \eqref{eq:defOPE} and study the asymptotic behavior, as $n\to \infty$, of the linear statistic
\begin{equation}
X^{(n)}_{f,\alpha,x_0}= \sum_{j=1}^n f(n^\alpha(x_j-x_0)).
\end{equation}
where $f: \bbR \rightarrow \bbC$ is a bounded function that decays at $\pm \infty$. The sum above is over the points of the ensemble, and the decay of $f$ at $\pm \infty$ implies that small weight is given to points that are at a distance greater than $\sim n^{-\alpha}$ from $x_0$.   

The proof of Theorems \ref{thm:JacobiCLT}, \ref{thm:main-result} and \ref{thm:CLTgeneral} follow from an analysis of the cumulants. The cumulants $\mathcal C^{(n)}_m(X^{(n)}_{f,\alpha,x_0})$ of the linear statistic $X^{(n)}_{f,\alpha,x_0}$ are defined by the following generating function
\begin{equation}
\log \bbE [{\rm e}^{{\rm i} t  X^{(n)}_{f,\alpha,x_0}}]= \sum_{m=1}^\infty t^m \mathcal C^{(n)}_m \left(X^{(n)}_{f,\alpha,x_0}  \right).
\end{equation}
It is not difficult to see that  the $m$-th  cumulant can be expressed in terms of the first $m$ moments and, vice versa, the $m$-th moment can be expressed in terms of the first $m$ cumulants.  Therefore in order to prove that the difference of the moments tends to zeros in Theorem \ref{th:comparison-general},  it is enough to show that the the difference of the cumulants tends to zero. To prove a  Central Limit Theorem we need to show that $$\mathcal C^{(n)}_m \left(X^{(n)}_{f,\alpha,x_0} \right)\to 0$$ for $m>2$ and that $\mathcal C_2^{(n)}\left( X^{(n)}_{f,\alpha,x_0}\right) $ has a limit as $n \to \infty$.

\subsection{Determinantal structure and a first rewriting} \label{sec:DetStruct}

It is well-known that the OPE of size $n$ with reference measure $\mu$ defines a determinantal point process whose kernel is the Christoffel-Darboux kernel associated with $\mu$:
\beq \no
K_n(x,y)=\sum_{j=0}^{n-1}p_j(x)p_j(y)
\eeq
where the $p_j$'s are the orthonormal polynomials associated with $\mu$. This is the kernel of the projection onto the space of polynomials of degree $\leq n-1$. The determinantal structure of the orthogonal polynomial ensemble says that
\beq \no
\mathbb{E}\left( \exp \left (t X^{(n)}_{f,\alpha,x_0} \right )\right)=\det \left(1+\left(e^{tf_{\alpha,x_0}^{(n)}}-1\right)K_n \right)_{\mathbb L^2(\mu)}
\eeq 
where we write 
\beq \no 
f_{\alpha,x_0}^{(n)}(x)= f(n^\alpha(x-x_0))
\eeq 
and the determinant on the right hand side is a Fredholm determinant of the integral operator on  $\mathbb L^2(\mu)$ with integral kernel $({\rm e}^{t f_{\alpha,x_0}^{(n)}(x)}-1)K_n(x,y)$, and $1$ stands for the identity operator. For an excellent review  on determinantal point processes (incuding a proof of this statement) we refer to \cite{J}. We will omit the notation $\mathbb L_2(\mu)$ in the index from now on. 

Since $\| f \|_\infty < \infty$ and $K_n$ is a bounded operator, we have that for $|t|$ sufficiently small 
\beq \no
\left \| \left( e^{tf_{\alpha,x_0}^{(n)}}-1 \right) K_n \right \|<1.
\eeq
Thus we may rewrite the determinant as a sum of traces
\beq \no
\begin{split}
\det(1+({\rm e}^{t f_{\alpha,x_0}^{(n)}}-1) K)&=\exp \Tr \log (1+({\rm e}^{t f_{\alpha,x_0}^{(n)}}-1) K)\\
&=\exp \sum_{j=1}^\infty \frac{(-1)^{j+1}}{j} \Tr \left( ({\rm e}^{t f_{\alpha,x_0}^{(n)}}-1) K\right)^j.
\end{split}
\eeq
By expanding the exponential in a Taylor series we obtain
\begin{multline}\no
\det(1+({\rm e}^{t f_{\alpha,x_0}^{(n)}}-1) K)\\=\exp \sum_{j=1}^\infty \frac{(-1)^{j+1}}{j}  \sum_{l_1,\ldots, l_j =1}^{\infty} t^{l_1+\cdots +l_j}\frac{\Tr  \left( f_{\alpha,x_0}^{(n)}\right)^{l_1} K \cdots \left(f_{\alpha,x_0}^{(n)} \right)^{l_j} K}{l_1!\cdots l_j!}, 
\end{multline}
which, by an extra reorganization, can be turned into
\begin{multline}\label{eq:standard:)}
\log \det(1+({\rm e}^{t f_{\alpha,x_0}^{(n)}}-1) K)\\
=\sum_{m=1}^\infty t^m\sum_{j=1}^m \frac{(-1)^{j+1}}{j}  \sum_{\overset{l_1+\cdots +l_j=m}{l_i\geq 1}} \frac{\Tr \left( f_{\alpha,x_0}^{(n)} \right)^{l_1} K \cdots  \left(f_{\alpha,x_0}^{(n)}\right)^{l_j} K}{l_1!\cdots l_j!}.
\end{multline}
By equating coefficients we have a standard formula for the cumulants of the linear statistic for a determinantal point process 
\beq \no
\begin{split}
&\mathcal{C}_m^{(n)}\left(X^{(n)}_{f,\alpha,x_0}  \right) \\
&\quad=\sum_{j=1}^{m}\frac{(-1)^{j+1}}{j} \sum_{l_1+\ldots+l_j=m, l_i \geq 1}\frac{\textrm{Tr}\left( f_{\alpha,x_0}^{(n)}\right)^{l_1}K_n\cdots \left( f_{\alpha,x_0}^{(n)}\right)^{l_j}K_n}{l_1! \cdots l_j!}.
\end{split}
\eeq 
We note that this formula was also the starting point in \cite{S} in the analysis of the fluctuations of mesosopic linear statistics  for the CUE. 

As in \cite{BD-Nevai} we will rewrite this formula by using the fact that for $m\geq 2$, 
\beq \no
\sum_{j=1}^m \frac{(-1)^{j+1}}{j}\sum_{l_1+\ldots+l_j=m, l_i\geq 1}\frac{1}{l_1!\cdots l_j!}=0,
\eeq 
(which follows from expanding the logarithm and exponential in $\log \left(1+\left(e^x-1 \right) \right)$),
to note that 
\beq \no
\sum_{j=1}^{m}\frac{(-1)^{j+1}}{j} \sum_{l_1+\ldots+l_j=m, l_i \geq 1}\frac{\textrm{Tr}\left(f_{\alpha,x_0}^{(n)} \right)^m K_n}{l_1! \cdots l_j!}=0.
\eeq 
This immediately implies that
\beq \label{eq:cumulant}
\begin{split}
&\mathcal{C}_m^{(n)}\left(X^{(n)}_{f,\alpha,x_0}  \right) \\
&\quad=\sum_{j=1}^{m}\frac{(-1)^{j+1}}{j} \sum_{l_1+\ldots+l_j=m, l_i \geq 1}\frac{\textrm{Tr}\left(f_{\alpha,x_0}^{(n)} \right)^{l_1}K_n\cdots \left(f_{\alpha,x_0}^{(n)} \right)^{l_j}K_n -\textrm{Tr}\left(f_{\alpha,x_0}^{(n)} \right)^m K_n}{l_1! \cdots l_j!},
\end{split}
\eeq 
for $m \geq 2$, which is the formula we want to use for the cumulants. The cancellation captured in this formula for the cumulants allowed us in \cite{BD-Nevai} to prove the following bound.

\begin{lemma}[Lemma 2.2 in \cite{BD-Nevai}] \label{lem:Concentration}
There exists a universal constant, $A>0$, such that for any bounded function $h: \bbR \rightarrow \bbR$ 
\beq \label{eq:boundonmoment}
\left|\log \bbE \left[{\rm e}^{ t (X_h^{(n)} - \bbE X_h^{(n)}) }\right]\right|\leq  A |t|^2 \Var X_h^{(n)},
\eeq
for $|t|\leq \frac{1}{3 \|h\|_\infty}$. 
\end{lemma}
Note that this Lemma also implies that for $m\geq 2$, we have $$\left|\mathcal{C}_m^{(n)}\left(X^{(n)}_{f,\alpha,x_0}  \right)\right| \leq c_m \Var X^{(n)}_{f,\alpha,x_0},$$
for some constant $c_m$ that does not depend on $n$ (in fact, in \cite{BD-Nevai} the latter was proved first,   obtaining \eqref{eq:boundonmoment} as a corollary). Since the cumulants and the moments can be expressed in terms of each other, the same statement holds when we replace the cumulants with the moments of $X^{(n)}_{f,\alpha,x_0}-\bbE X^{(n)}_{f,\alpha,x_0}$. 

We will  need the following simple corollary on the continuity of the generating function (and hence the moments and the cumulants) as a function of $h$. 
\begin{lemma}\label{lem:continuitymoments} There exists a universal constant, $C>0$, such that for any bounded functions $g: \bbR \rightarrow \bbR$  and $h: \bbR \rightarrow \bbR$ 
\begin{multline}
\left|\bbE \left[{\rm e}^{ t (X_g^{(n)} - \bbE X_g^{(n)})} \right]-\bbE \left[{\rm e}^{ t (X_h^{(n)} - \bbE X_h^{(n)})} \right]\right|\\
\leq |t|(\Var X_{g-h}^{(n)})^{1/2} {\rm e}^{C |t|^2\left(\Var X_h^{(n)}+\Var X_{g-h}^{(n)}\right)},
\end{multline}
for $|t|\leq \min( \frac{1}{6 \|g\|_\infty},\frac{1}{6 \|h\|_\infty})$. 
\end{lemma}
\begin{proof}
We will use the notation 
$$R_n(f)= \bbE \left[{\rm e}^{  (X_f^{(n)} - \bbE X_f^{(n)})}\right].$$
We first write
\begin{multline}\label{eq:Rntming}
|R_n(t h)- R_n(t g)|=\left|\int_0^t \frac{{\rm d} }{{\rm d} s}  R_n(t h+s(g-h)) {\rm d} s\right|\\
 \leq \int_0^{t}  \left| \frac{{\rm d} }{{\rm d} s}  R_n(t h+s(g-h)) \right| {\rm d} s.
 \end{multline}
Now note, by the linearity of $f \mapsto X_f^{(n)}$, 
$$\frac{{\rm d} }{{\rm d} s}  R_n(t h+s(g-h))= \bbE \left[\left(X_{g-h}^{(n)} - \bbE X_{g-h} ^{(n)})\right){\rm e}^{ t (X_h^{(n)} - \bbE X_h^{(n)})+s (X_{g-h}^{(n)} - \bbE X_{g-h}^{(n)})} \right],$$
which, by applying the Cauchy-Schwarz inequality, implies that 
\beq \label{eq:Rnderiv} 
\left|\frac{{\rm d} }{{\rm d} s}  R_n(t h+s(g-h))\right| \leq \sqrt{\Var X_{h-g}^{(n)}} \sqrt{R_n( 2 \Re t h+  2 \Re s (g-h))}.
\eeq 
Moreover, by \eqref{eq:boundonmoment} and $2ab \leq |a|^2+|b|^2$ for $a,b \in \bbR$,
\begin{equation} \label{eq:rnths} 
|R_n(t h + s(g-h))| \leq {\rm e}^{2 A |t|^2 \Var X_{ h }^{(n)}+ 2 |s|^2 \Var X_{(g-h)} ^{(n)}}.
\end{equation}
By inserting \eqref{eq:rnths} in \eqref{eq:Rnderiv} and applying the result to \eqref{eq:Rntming} we get the statement.
\end{proof}

\subsection{Cumulants in terms of the Jacobi operator} \label{subsec:Jacobi}

The consideration in the previous paragraph are not special for OPE's but hold for general determinantal point processes that have a kernel that is a self-adjoint projection, see \cite{BD-Nevai}.  The additional structure for OPE's is exploited in \cite{BD-CLT} by using the spectral theorem to replace functions in \eqref{eq:cumulant} with matrices. Thus we have
\beq \label{eq:cumulant-spectral-form1}
\begin{split}
&\mathcal{C}_m^{(n)}\left(X^{(n)}_{f,\alpha,x_0}  \right) \\
&\quad=\sum_{j=1}^{m}\frac{(-1)^{j+1}}{j} \sum_{l_1+\ldots+l_j=m, l_i \geq 1}\frac{\textrm{Tr}\left(f_{\alpha,x_0}^{(n)}(J) \right)^{l_1}P_n\cdots \left(f_{\alpha,x_0}^{(n)}(J) \right)^{l_j}P_n -\textrm{Tr}\left(f_{\alpha,x_0}^{(n)} (J)\right)^m P_n}{l_1! \cdots l_j!},
\end{split}
\eeq 
where now $J$ is the Jacobi matrix associated with $\mu$ acting on $\ell_2(\bbN)$, and $P_n$ is the projection onto the first $n$ coordinates, i.e.,
\beq \label{eq:defPn}
P_n v (j)=\left \{ \begin{array}{cc} v(j) & j \leq n \\
0 & \textrm{otherwise.} \end{array} \right.
\eeq
The trace in \eqref{eq:cumulant-spectral-form1} is taken in $\ell^2$. The benefit of \eqref{eq:cumulant-spectral-form1} is the fact that $J$ is banded (in fact, tri-diagonal) and therefore so is $f(J)$ whenever $f$ is a polynomial. For banded $f(J)$ it is simple to see that the cumulants only depend on a finite number of entries of $f(J)$ and, since $f$ is a polynomial, only on a finite number of entries $J$. Moreover, the number of relevant entries of $J$ does not depend in $n$. A comparison result similar to Theorem \ref{th:comparison-general} follows from these considerations almost immediately. To prove a universal Central Limit Theorem it is then enough to deal with an OPE for a special choice of $\mu$ for which a Central Limit Theorem is relatively straightforward.

This strategy, however, breaks down in the mesoscopic scales, since polynomials are not suitable test functions for these scales. To get around this, we turn to the resolvent of $J$. More precisely, for any fixed $N \in \bbN$, $\eta_1,\ldots,\eta_N \in \{\eta \in \bbC \mid \Im \eta>0\}$ and $c_1,\ldots,c_N \in \bbR$ we want to consider
\beq \label{eq:specialclass}
f(x)=f_{\eta_1,\ldots,\eta_N;c_1,\ldots,c_N}(x)=\sum_{j=1}^N c_j \Im\frac{1}{x-\eta_j} \quad \Im \eta_j >0.
\eeq

To write the cumulants of $X_{f,\alpha,x_0}^{(n)}$ for such $f$ using \eqref{eq:cumulant-spectral-form1} we use the following notation. For a Jacobi matrix, $J$, and $\lambda \in \bbC \setminus \bbR$, let
\beq \label{eq:defgreen}
G(\lambda)=G(J;\lambda)=\left(J-\lambda \right)^{-1}
\eeq
and 
\beq \label{eq:defImaggreen}
F(\lambda)=F(J;\lambda)=\Im G(J;\lambda)=\frac{1}{2i}\left( G(J;\lambda)-G(J;\overline{\lambda} )\right)
\eeq
where $J$ will be suppressed whenever there is no danger of confusion. We will also write 
\beq \label{eq:lincomImaggreen} F(\vec \lambda, \vec c) = \sum_{j=1}^N c_j F(\lambda_j)
\eeq
Letting $\lambda_j=x_0+\frac{\eta_j}{n^\alpha}$ we get from \eqref{eq:cumulant-spectral-form1} that
\beq \label{eq:cumulant-spectral-form}
\begin{split}
&\mathcal{C}_m^{(n)}\left(X_{f,\alpha,x_0}^{(n)} \right) \\
&\quad =\sum_{j=1}^{m}\frac{(-1)^{j+1}}{j n^{\alpha m}}\sum_{l_1+\ldots+l_j=m, l_i \geq 1}\frac{\textrm{Tr}F(\vec \lambda, \vec c)^{l_1}P_n\cdots F(\vec \lambda,\vec c)^{l_j}P_n -\textrm{Tr}F(\vec \lambda_n,\vec c_j )^m P_n}{l_1! \cdots l_j!}.
\end{split}
\eeq 
This is the starting point of our analysis in the next section. 

The benefit of working with  the resolvent is the following. Although $F(\vec \lambda, \vec c)$ is not banded, its diagonals decay exponentially in their distance from the main diagonal. This fact, known as the \emph{Combes-Thomas estimate}, holds for multidimensional Schr\"odinger operators as well and has extensive applications in spectral theory. Its method of proof goes back to a general argument of Combes and Thomas \cite{CT}. We present it here in the form most suited for our purposes, with a proof for completeness.

\begin{proposition} \label{prop:CT}
Let 
\beq \no
J=
\begin{pmatrix}
b_1 & a_1 & 0 & \dots \\
a_1 & b_2 & a_2 & \dots \\
0 & a_2 & b_3 & \dots \\
\vdots & \vdots & \vdots & \ddots
\end{pmatrix}
\eeq
be a Jacobi matrix with  $b_n \in \bbR$ and $0<a_n \leq A$ for some constant $A>0$. Then for any $\lambda \in \bbC$ with ${\rm Im} \lambda>0$ and any $n,m \in \bbN$
\beq \label{eq:CT-estimate}
\left| \left (G(\lambda) \right)_{n,m} \right| \leq \frac{2}{{\rm Im} z}e^{-\min \left(1, \frac{{\rm Im z}}{4e A}\right) |n-m|}.
\eeq
\end{proposition}
\begin{proof}
Since $\left(G(\lambda) \right)_{n,m}=\left(G(\lambda) \right)_{m,n}$, we may assume without loss of generality that $n \leq m$. Let  $\theta =\min\left(1, \frac{{\rm Im} z}{4eA} \right)$ and define the diagonal matrix $R$ by
\beq \no
R_{j,j}=e^{\theta j}.
\eeq 
It follows that 
\beq \no
e^{\theta |n-m|} \left( G(\lambda) \right)_{n,m}=\left( R^{-1} G(\lambda) R \right)_{n,m},
\eeq
which implies that
\beq \label{eq:CT1}
\left|   \left( G(\lambda) \right)_{n,m} \right | \leq e^{-\theta |n-m|} \left \|  R^{-1} G(\lambda) R \ \right \|_\infty,
\eeq
where $\| \cdot \|_\infty$ is the operator norm.
Now note that
\beq \no
R^{-1}G(\lambda) R=R^{-1}\left(J-\lambda \right)^{-1} R=\left(R^{-1}  \left(J-\lambda \right) R\right)^{-1}
=\left( R^{-1} J R-\lambda \right)^{-1}
\eeq
so that, by the resolvent formula $A^{-1}-B^{-1}=A^{-1}\left(B-A \right)B^{-1}$,
\beq \no
R^{-1}G(\lambda) R=G(\lambda)+R^{-1}G(\lambda) R \left(J-R^{-1}JR \right) G(\lambda).
\eeq
Thus 
\beq \label{eq:CT2}
\left \| R^{-1}G(\lambda) R \right \|_\infty \leq \left \|G(\lambda) \right \|_\infty+ \left \|R^{-1}G(\lambda) R \right \|_\infty \left \| J-R^{-1}JR \right \|_\infty \left \| G(\lambda) \right \|_\infty.
\eeq
To finish the proof we note that
$
\left \| J-R^{-1}JR \right \|_\infty \leq 2 A \left|e^{\theta}-1 \right| \leq 2A e \theta 
$ and that, since $J$ is self-adjoint, $\left \| G(\lambda) \right \|_\infty \leq  \left( {\rm Im} \lambda \right)^{-1}$. Plugging these facts into \eqref{eq:CT2} we get
\beq \no
\left \| R^{-1}G(\lambda) R \right \|_\infty \leq \left( {\rm Im} \lambda \right)^{-1}+ \left \|R^{-1}G(\lambda) R \right \|_\infty \frac{2eA \theta}{ {\rm Im} \lambda },
\eeq
which implies, since $\frac{2eA \theta}{ {\rm Im} \lambda } \leq 1/2$, that
\beq \no
\left \| R^{-1}G(\lambda) R \right \|_\infty \leq \frac{2}{  {\rm Im} \lambda }.
\eeq
Plugging this into \eqref{eq:CT1} finishes the proof.
\end{proof}

The estimate \eqref{eq:CT-estimate} means that $F(\vec \lambda,\vec v)$ is, in a sense, approximately banded. It is this property that allows us to apply the strategy of \cite{BD-CLT} to $f$ as above. We demonstrate this in the next section.

\subsection{Further notation} \label{sec:operatorStuff}
We end this section by  discussing some notation that we will use. If $A$ is an operator on a (seperable) Hilbert space we let \\
(i) $\|A\|_\infty$ be the operator norm,\\
(ii) $\|A\|_1$ be the trace norm,\\
(iii) $\|A\|_2$ be the Hilbert-Schmidt norm.\\
It is standard that 
$$ \|AB\|_j= \|A\|_j \|B\|_\infty, \qquad  \|AB\|_j= \|A\|_\infty \|B\|_j,$$
for $j=1,2, \infty$. Moreover, 
$$\|AB\|_1\leq \|A\|_2 \|B\|_2.$$
It $A$ is trace class then 
$$|\Tr A| \leq \|A\|_1,$$
and 
$$\left|\det (I+A)-1\right|\leq \|A\|_1 \exp( \|A\|_1+1),$$
where the left-hand side is the Fredholm determinant. For further details on trace class and Hilbert-Schmidt operators we refer to \cite{Simon}.

In the upcoming analysis we will frequently use the following estimates: if $A= \left(A_{i,j}\right)_{i,j=1}^\infty$ is an infinite matrix then, viewed as an operator on $\ell_2(\bbN)$, we have 
\begin{align*}
\|A\|_\infty &\leq \sum_{j=-\infty}^\infty \sup_k |A_{k,k+j}|,\\
\|A\|_1 &\leq \sum_{i,j=1}^\infty |A_{i,j}|\\
\|A\|_2 &=\left( \sum_{i,j=1}^\infty |A_{i,j}|^2\right)^{1/2}.
\end{align*}

%%%%%%%%%%%%%%%%%%%%%%%%%%%%%%%%%%%%%%%%%%%%%%% 

%%%%%%%%%%%%%%%%%%%%%%%%%%%%%%%%%%%%%%%%%%%%%%%%%%%%%%%%%%%

\section{Universality of the cumulants} \label{sec:comparison}
In this section we discuss a comparison principle for the cumulants for mesoscopic linear statistics for two different OPE's. In particular, we will prove Theorem \ref{thm:main-result} for functions of the form \ref{eq:specialclass}. The proof of the complete statement then follows by extending this class of function to $C^1_c(\bbR)$, which  will be done in Section \ref{sec:extension}.
\subsection{A general comparison principle}

Since we want to apply the comparison principle in a slightly more general setting than that of resolvents, we first formulate precisely the hypothesis we need.

\begin{hypothesis}\label{Hypothesis1}
Given $0<\alpha<1$, we focus on matrix sequences, $\{\mathcal A^{(n)}\}_{n=1}^\infty$, satisfying the following two conditions: \\
{\bf 1.} The matrices $\{\mathcal A^{(n)}\}_{n=1}^\infty$ are uniformly bounded in operator norm. \\  
{\bf 2.} For some constants $C, C'>0$, independent of $n$,
\beq \label{eq:alpha-exp-decay}
\left | \mathcal A^{(n)}_{r,s}\right| \leq C e^{-\frac{C'|r-s|}{n^\alpha}}.
\eeq
We shall say that a matrix sequence satisfies Hypothesis \ref{Hypothesis1} (for $\alpha$), if these two conditions are satisfied.  
\end{hypothesis}

 For a matrix $\mathcal{A}$ we shall abuse notation and write
\beq  \label{eq:newcumu}
\calC_m^{(n)}(\mathcal A) = \sum_{j=2}^m \frac{(-1)^j}{j} \sum_{\overset{l_1+ l_2+ \cdots l_j=m}{ l_i\geq 1}} \frac{\textrm{Tr} \mathcal A^{l_1} P_n\cdots \mathcal A^{l_j} P_n-\textrm{Tr} \mathcal A^m P_n }{l_1! \cdots l_j!}. 
\eeq
Finally, for any two positive numbers $\ell_1<\ell_2$, we let
\beq \no
P_{\ell_1,\ell_2}=P_{[\ell_2]+1}-P_{[\ell_1]-1},
\eeq
where $[x]$ is the integer value of $x$. We can now state the main technical result of this paper.

\begin{theorem} \label{th:comparison-general}
Fix $0<\alpha<1$ and let $\{ \mathcal A^{(n)} \}_{n=1}^\infty$ $\{ \mathcal B^{(n)} \}_{n=1}^\infty$ be two matrix sequences, each satisfying  Hypothesis \ref{Hypothesis1} with $\alpha$. 
Then, for any $\beta>\alpha$ and any $m \geq 2$, 
\beq \label{eq:almost-band}
\begin{split}
& \left|\calC_m^{(n)} \left(\mathcal B^{(n)} \right) - \calC_m^{(n)} \left( \mathcal A^{(n)} \right) \right| \\
&\quad \leq C(m,\beta) \left \| P_{\left \{n-2mn^\beta,n+2mn^\beta \right\}} \left( \calB^{(n)} -\calA^{(n)}\right) P_{\left \{n-2mn^\beta,n+2mn^\beta\right \}} \right \|_1+o(1)
\end{split}
\eeq
as $n \rightarrow \infty$, where $C(m,\beta)$ is a constant that depends on $m$ and $\beta$ but is independent of $n$.
\end{theorem}

\begin{proof}

We may assume that the constants in Hypothesis \ref{Hypothesis1} for $\calA^{(n)}$ and $\calB^{(n)}$ are the same.
 
First, write
\beq \no
\calA^{(n)}=\calA^{(n)}_1+\calA^{(n)}_2,
\eeq
\beq \no
\calB^{(n)}=\calB^{(n)}_1+\calB^{(n)}_2
\eeq
where
\beq \no
\left( \calA^{(n)}_1 \right)_{r,s}=\left \{ \begin{array}{cc}
\calA^{(n)}_{r,s} & |r-s| \leq n^\beta \\
0 & \textrm{otherwise}
\end{array} \right.
\eeq
and similarly for $\calB^{(n)}$, so that $\calA^{(n)}_1$ and $\calB^{(n)}_1$ are banded matrices of  band size $2n^\beta$.
By Hypothesis \ref{Hypothesis1} 
\beq \label{eq:exponent-boundA}
\left \| \left(\calA^{(n)}_2 \right)_{r,s}\right \|_\infty \leq \sum_{j=n^\beta}^\infty C e^{-\frac{C' j}{n^\alpha}}\leq \widetilde{C}n^\alpha e^{-C' n^{\beta-\alpha}}
\eeq 
and
\beq \label{eq:exponent-boundB}
\left \| \left(\calB^{(n)}_2 \right)_{r,s}\right \| _\infty\leq \sum_{j=n^\beta}^\infty C e^{-\frac{C' j}{n^\alpha}}\leq \widetilde{C}n^\alpha e^{-C' n^{\beta-\alpha}}
\eeq

Now write each term in the sum in \eqref{eq:newcumu} as
\beq \no
 \frac{\textrm{Tr}\left( \calA^{l_1}_1+\calA^{l_1}_2 \right)P_n\cdots \left( \calA^{l_j}_1 +\calA^{l_j}_2 \right) P_n-\textrm{Tr}\left( \calA^m_1+\calA^m_2 \right) P_n }{l_1! \cdots l_j!}
\eeq
After opening the paranthesis and collecting terms we get
\beq \no
\begin{split}
\calC_m^{(n)}(\mathcal A) &= \sum_{j=2}^m \frac{(-1)^j}{j} \sum_{\overset{l_1+ l_2+ \cdots l_j=m}{ l_i\geq 1}} \frac{\textrm{Tr} \left(\calA_1\right)^{l_1} P_n\cdots \left(\calA_1\right)^{l_j} P_n-\textrm{Tr} \left(\calA_1\right)^m P_n }{l_1! \cdots l_j!} \\
&\qquad+r_n\left(\calA\right). 
\end{split}
\eeq
were $r_n\left(\calA \right)$ is a sum of terms, each of which contains a factor of $\left( \calA_2 \right)^{l_k}$ for some $l_k \leq m$. 
The number of terms in the sum is a constant that depends solely on $m$, all matrices involved are bounded, and so, since the trace is over an $n$-dimensional subspace, we get from \eqref{eq:exponent-boundA}
\beq \no
\left \| r_n\left( \calA \right) \right \|_\infty\leq \widetilde{C}(m) n^{1+\alpha} e^{-C' n^{\beta-\alpha}} \rightarrow 0
\eeq
as $n \rightarrow \infty$. The same can be done for $\calC_m^{(n)}(\calB)$. 

Thus we see that up to an $o(1)$ term, $\calC_m^{(n)}(\calB)-\calC_m^{(n)}(\calA)$ is a sum of terms of the form
\beq \label{eq:summand-comparison}
\begin{split}
& \left( \Tr \left(\calB_1^{(n)} \right)^{l_1}P_n\cdots P_n\left( \calB_1^{(n)}\right)^{l_j}P_n-\Tr \left(\calA_1^{(n)} \right)^{l_1}P_n\cdots P_n \left( \calA_1^{(n)}\right)^{l_j}P_n \right)\\
&\quad- \left( \Tr \left( \calB_1^{(n)} \right)^mP_n-\Tr \left( \calA_1^{(n)} \right)^mP_n \right)
\end{split}
\eeq
where $2 \leq j \leq m$, and $\{ l_1, l_2,\ldots, l_j \}$ are positive integers satisfying $l_1+l_2+\ldots+l_j=m$.

Writing the first difference as a telescoping sum
\beq \label{eq:telescope1}
\begin{split}
&  \left(\calB_1^{(n)} \right)^{l_1}P_n\cdots \left( \calB_1^{(n)}\right)^{l_j}P_n-\left(\calA_1^{(n)} \right)^{l_1}P_n\cdots \left( \calA_1^{(n)}\right)^{l_j}P_n \\
&= \left(\calB_1^{(n)} \right)^{l_1}P_n\cdots \left( \calB_1^{(n)}\right)^{l_j}P_n\\
&\quad-\left (\calA_1^{(n)} \right)\left(\calB_1^{(n)} \right)^{l_1-1}P_n\cdots \left( \calB_1^{(n)}\right)^{l_j}P_n \\
&\quad+ \left (\calA_1^{(n)} \right)\left(\calB_1^{(n)} \right)^{l_1-1}P_n\cdots \left( \calB_1^{(n)}\right)^{l_j}P_n \\
&\quad -\left (\calA_1^{(n)} \right)^2\left(\calB_1^{(n)} \right)^{l_1-2}P_n\cdots \left( \calB_1^{(n)}\right)^{l_j}P_n \\
&\quad+ \ldots \\
&\quad+\left(\calA_1^{(n)} \right)^{l_1}P_n\cdots \left( \calA_1^{(n)}\right)^{l_j-1} \left( \calB_1^{(n)}\right)P_n \\
&\quad-\left(\calA_1^{(n)} \right)^{l_1}P_n\cdots \left( \calA_1^{(n)}\right)^{l_j}P_n
\end{split}
\eeq
and similarly for the second term
\beq \label{eq:telescope2}
\begin{split}
 \left( \calB_1^{(n)} \right)^mP_n- \left( \calA_1^{(n)} \right)^mP_n &=  \left( \calB_1^{(n)} \right)^mP_n -\left( \calA_1^{(n)} \right)\left( \calB_1^{(n)} \right)^{m-1}P_n \\
&\quad+ \ldots \left( \calA_1^{(n)} \right)^{m-1} \calB_1^{(n)} P_n-\left( \calA_1^{(n)} \right)^mP_n,
\end{split}
\eeq
we see that \eqref{eq:summand-comparison} is a sum of $m$ terms of the form
\beq \label{eq:terms1}
\begin{split}
& \Tr \left( \calA_1^{(n)} \right)^{l_1}P_n \cdots P_n \left( \calA_1^{(n)} \right)^{l_k-i}\left( \calB_1^{(n)}\right)^i P_n \cdots \left( \calB_1^{(n)} \right)^{l_j}P_n \\
&\quad -\Tr \left( \calA_1^{(n)} \right)^{l_1}P_n \cdots P_n \left( \calA_1^{(n)} \right)^{l_k+1-i}\left( \calB_1^{(n)}\right)^{i-1} P_n \cdots \left( \calB_1^{(n)} \right)^{l_j}P_n \\
&\quad - \Tr \left( \calA_1^{(n)} \right)^{l_1+\ldots+ l_k-i} \left( \calB_1^{(n)} \right)^{m-l_1\ldots-l_k+i}P_n \\
&\quad +\Tr  \left( \calA_1^{(n)} \right)^{l_1+\ldots+ l_k+1-i} \left( \calB_1^{(n)} \right)^{m-l_1\ldots-l_k+i-1}P_n \\
&= \Tr \left( \left( \calA_1^{(n)} \right)^{l_1}P_n \cdots \left( \calB_1^{(n)} -\calA_1^{(n)}\right) \quad \left( \calB_1^{(n)} \right)^{i-1} P_n \cdots \left( \calB_1^{(n)} \right)^{l_j}P_n \right)\\
&\quad -\Tr \left(  \left( \calA_1^{(n)} \right)^{l_1+\ldots+ l_k-i} \left( \calB_1^{(n)}-\calA_1^{(n)} \right) \left(\calB_1^{(n)}\right)^{m-l_1\ldots-l_k+i-1}P_n \right).
\end{split}
\eeq

Now, for any two matrices, $A$ and $B$, 
\beq \no
\Tr \left( AP_nBP_n\right)-\Tr\left( ABP_n\right)=\Tr \left( A(I-P_n)BP_n\right)=\sum_{s=1}^n \sum_{r>n}A_{s,r}B_{r,s}.
\eeq
Extending this fact to a product of $j$ matrices we get that
\beq \label{eq:help1}
\begin{split}
&\Tr \left( \left( \calA_1^{(n)} \right)^{l_1}P_n \cdots \left( \calB_1^{(n)} -\calA_1^{(n)}\right) \quad \left( \calB_1^{(n)} \right)^{i-1} P_n \cdots \left( \calB_1^{(n)} \right)^{l_j}P_n \right)\\
&\quad -\Tr \left(  \left( \calA_1^{(n)} \right)^{l_1+\ldots+ l_k-i} \left( \calB_1^{(n)}-\calA_1^{(n)} \right) \left(\calB_1^{(n)}\right)^{m-l_1\ldots-l_k+i-1}P_n \right)\\
&=\sum_{s=1}^n \sum_{\left(r_1,r_2,\ldots, r_{j-1}\right) \in I_{j,n}}  \left(\left( \calA_1^{(n)} \right)^{l_1}\right)_{s,r_1} \cdots \\
&\quad \times \left( \left( \calA_1^{(n)} \right)^{l_k-i}\left( \calB_1^{(n)} -\calA_1^{(n)}\right) \left( \calB_1^{(n)} \right)^{i-1}\right)_{r_{k-1},r_k} \cdots \left( \left( \calB_1^{(n)} \right)^{l_j}\right)_{r_{j-1},s}
\end{split}
\eeq
where 
\beq \no
I_{j,n}=\left \{ \left(r_1,r_2,\ldots, r_{j-1} \right) \in \bbN^{j-1} \mid \exists q : r_q>n \right \}.
\eeq

However, recall that $\calA_1$ and $\calB_1$ are banded with band size $n^\beta$. Thus, by the definition of $I_{j,n}$ and since there are $m$ matrices in the product, we see that
\beq \no
\begin{split}
&\sum_{s=1}^{[n-mn^\beta]} \sum_{\left(r_1,r_2,\ldots, r_{j-1}\right) \in I_{j,n}}  \left(\left( \calA_1^{(n)} \right)^{l_1}\right)_{s,r_1} \cdots \\
&\quad \times \left( \left( \calA_1^{(n)} \right)^{l_k-i}\left( \calB_1^{(n)} -\calA_1^{(n)}\right) \left( \calB_1^{(n)} \right)^{i-1}\right)_{r_{k-1},r_k} \cdots \left( \left( \calB_1^{(n)} \right)^{l_j}\right)_{r_{j-1},s}=0
\end{split}
\eeq
so that
\beq \label{eq:help2}
\begin{split}
&\Tr \left( \left( \calA_1^{(n)} \right)^{l_1}P_n \cdots \left( \calB_1^{(n)} -\calA_1^{(n)}\right) \quad \left( \calB_1^{(n)} \right)^{i-1} P_n \cdots \left( \calB_1^{(n)} \right)^{l_j}P_n \right)\\
&\quad -\Tr \left(  \left( \calA_1^{(n)} \right)^{l_1+\ldots+ l_k-i} \left( \calB_1^{(n)}-\calA_1^{(n)} \right) \left(\calB_1^{(n)}\right)^{m-l_1\ldots-l_k+i-1}P_n \right)\\
&=\sum_{s=[n-mn^\beta]}^n \sum_{\left(r_1,r_2,\ldots, r_{j-1}\right) \in I_{j,n}}  \left(\left( \calA_1^{(n)} \right)^{l_1}\right)_{s,r_1} \cdots \\
&\quad \times \left( \left( \calA_1^{(n)} \right)^{l_k-i}\left( \calB_1^{(n)} -\calA_1^{(n)}\right) \left( \calB_1^{(n)} \right)^{i-1}\right)_{r_{k-1},r_k} \cdots \left( \left( \calB_1^{(n)} \right)^{l_j}\right)_{r_{j-1},s}.
\end{split}
\eeq

Applying \eqref{eq:help1} again we get
\beq \no
\begin{split}
&\Tr \left( \left( \calA_1^{(n)} \right)^{l_1}P_n \cdots \left( \calB_1^{(n)} -\calA_1^{(n)}\right) \quad \left( \calB_1^{(n)} \right)^{i-1} P_n \cdots \left( \calB_1^{(n)} \right)^{l_j}P_n \right)\\
&\quad -\Tr \left(  \left( \calA_1^{(n)} \right)^{l_1+\ldots+ l_k-i} \left( \calB_1^{(n)}-\calA_1^{(n)} \right) \left(\calB_1^{(n)}\right)^{m-l_1\ldots-l_k+i-1}P_n \right)\\
&=\Tr P_{\left \{n-mn^\beta,n \right \}} \Big( \left( \calA_1^{(n)} \right)^{l_1}P_n \cdots \left( \calB_1^{(n)} -\calA_1^{(n)}\right) \left( \calB_1^{(n)} \right)^{i-1} P_n \cdots \left( \calB_1^{(n)} \right)^{l_j} \\
&\quad-  \left( \calA_1^{(n)} \right)^{l_1+\ldots+ l_k-i} \left( \calB_1^{(n)}-\calA_1^{(n)} \right) \left(\calB_1^{(n)}\right)^{m-l_1\ldots-l_k+i-1} \Big),
\end{split}
\eeq
where recall that $P_{\ell_1,\ell_2}=P_{[\ell_2]+1}-P_{[\ell_1]-1}$.

To summarize, we have shown  $\calC_m^{(n)}\left(\calB^{(n)}\right)-\calC_m^{(n)}\left(\calA^{(n)} \right)$ is a sum of terms of the form
\beq \no
\Tr P_{\left \{n-mn^\beta,n \right \}} \Big( \widetilde{A} \left( \calB_1^{(n)} -\calA_1^{(n)}\right) \widetilde{B}
- \widetilde{C} \left( \calB_1^{(n)}-\calA_1^{(n)} \right) \widetilde{D} \Big)
\eeq
where $\widetilde{A},\widetilde{B},\widetilde{C},\widetilde{D}$ are all banded matrices of band size at most $mn^\beta$ and the number of terms in the sum depends only on $m$.

However, note that if $\widetilde{A}$ is a banded matrix with band size $mn^\beta$, then 
\beq \no
P_{\left\{b-mn^\beta,n \right\}}\widetilde{A} \left(I-P_{\left\{ n-2mn^\beta,n+2mn^\beta\right \}} \right)=0.
\eeq
This immediately implies that
\beq \no
\begin{split}
& \Tr P_{\left \{n-mn^\beta,n \right \}} \Big( \widetilde{A} \left( \calB_1^{(n)} -\calA_1^{(n)}\right) \widetilde{B}
- \widetilde{C} \left( \calB_1^{(n)}-\calA_1^{(n)} \right) \widetilde{D} \Big)\\
&=  \Tr P_{\left \{n-mn^\beta,n \right \}}  \widetilde{A}P_{\left\{ n-2mn^\beta,n+2mn^\beta\right \}} \left( \calB_1^{(n)} -\calA_1^{(n)}\right)P_{\left\{ n-2mn^\beta,n+2mn^\beta\right \}} \widetilde{B}\\
&- \Tr P_{\left \{n-mn^\beta,n \right \}}\widetilde{C} P_{\left\{ n-2mn^\beta,n+2mn^\beta\right \}} \left( \calB_1^{(n)}-\calA_1^{(n)} \right)P_{\left\{ n-2mn^\beta,n+2mn^\beta\right \}} \widetilde{D}. 
\end{split}
\eeq

Putting it all together and recalling that $\calA^{(n)}$ and $\calB^{(n)}$ are uniformly bounded, we see that 
\beq \no
\begin{split}
& \left|\calC_m^{(n)} \left(\calB^{(n)} \right) - \calC_m^{(n)} \left( \calA^{(n)} \right) \right| \\
&\quad \leq C(m) \left \| P_{\left \{ n-2mn^\beta,n+2mn^\beta \right \}} \left( \calB^{(n)} -\calA^{(n)}\right) P_{\left \{ n-2mn^\beta,n+2mn^\beta \right \}} \right \|_1+o(1).
\end{split}
\eeq
As this is \eqref{eq:almost-band}, we are done.

\end{proof}

The `almost-banded' structure of matrices satisfying Hypothesis \ref{Hypothesis1}  implies also

\begin{lemma} \label{lem:C-T-norm-estimate}
Suppose that $\{\calB^{(n)}\}_{n=1}^\infty$ satisfies Hypothesis \ref{Hypothesis1}. Then, for any $\beta>\alpha$ and any $m=1,2,3,\ldots$ there exist constants $D$ and $D'$ such that
\beq \label{eq:C-T-norm-estimate}
\left \|\left(I-P_{\left\{n-4mn^\beta,n+4mn^\beta \right\}} \right)\calB^{(n)} P_{\left\{ n-2mn^\beta,n+2mn^\beta \right\}} \right \|_1 \leq D n^{\beta+\alpha} e^{-D' n^{\beta-\alpha}},
\eeq 
\beq \label{eq:C-T-norm-estimate2}
\left \| P_{\left\{ n-2mn^\beta,n+2mn^\beta \right\}} \calB^{(n)}\left(I-P_{\left\{n-4mn^\beta,n+4mn^\beta \right\}} \right)\right \|_1 \leq D n^{\beta+\alpha} e^{-D'n^{\beta-\alpha}},
\eeq
\beq \label{eq:C-T-norm-estimate-HS}
\left \|\left(I-P_{\left\{n-4mn^\beta,n+4mn^\beta \right\}} \right)\calB^{(n)} P_{\left\{ n-2mn^\beta,n+2mn^\beta \right\}} \right \|_2 \leq D n^{\beta+\alpha} e^{-D' n^{\beta-\alpha}},
\eeq
and 
\beq \label{eq:C-T-norm-estimate2-HS}
\left \| P_{\left\{ n-2mn^\beta,n+2mn^\beta \right\}} \calB^{(n)}\left(I-P_{\left\{n-4mn^\beta,n+4mn^\beta \right\}} \right)\right \|_2 \leq D n^{\beta+\alpha} e^{-D'n^{\beta-\alpha}}
\eeq
\end{lemma}

\begin{proof}
We compute
\beq \no
\begin{split}
&\left \|\left(I-P_{\left\{n-4mn^\beta,n+4mn^\beta \right\}} \right)\calB^{(n)} P_{\left\{ n-2mn^\beta,n+2mn^\beta \right\}} \right \|_1  \\ 
& \quad \leq \sum_{n-2mn^\beta \leq i \leq n+2mn^\beta  } \sum_{1 \leq j \leq n-4mn^\beta, n+4mn^\beta \leq j <\infty}\left | \calB^{(n)} _{i,j} \right| \\
& \quad \leq 4mn^\beta \left( \sup_{n-2mn^\beta \leq i \leq n+2mn^\beta} \sum_{1 \leq j \leq n-4mn^\beta, n+4mn^\beta \leq j <\infty}\left | \calB^{(n)} _{i,j} \right| \right) \\
&\quad \leq 4mn^\beta \left(2\sum_{j=2m[n^\beta]}^\infty C e^{-\frac{-C'j}{n^\alpha}}   \right) \\
&\quad \leq 8Cmn^\beta e^{-2mn^{\beta-\alpha}}\frac{1}{1-e^{-\frac{C'}{n^\alpha}}} \\
&\quad \leq 16 \frac{C}{C'}mn^{\beta+\alpha} e^{-2mn^{\beta-\alpha}}
\end{split}
\eeq
for sufficiently large $n$ (depending on $\alpha$ and $C'$). This proves \eqref{eq:C-T-norm-estimate}. The proof of \eqref{eq:C-T-norm-estimate2} is the same. The same computation, with $ \sum \left | \calB^{(n)} _{i,j} \right| $ replaced by $\sum \left | \calB^{(n)} _{i,j} \right|^2 $ proves \eqref{eq:C-T-norm-estimate-HS} and \eqref{eq:C-T-norm-estimate2-HS}. 
\end{proof}

%%%%%%%%%%%%%%%%%%%%%%%%%%%%%%%%%

\subsection{Proof of Theorem \ref{thm:main-result} in case $f(x)=\sum_{j=1}^n c_j \Im \frac{1}{\eta_j-x}$}

We first formulate conditions on a perturbation of a Jacobi matrix under which cumulant asymptotics are preserved.

\begin{proposition} \label{Jacobi-perturbation}
Let $J_0$ and $J$ be two bounded Jacobi matrices and fix $\eta$ with $\rm{Im}\eta>0$ and $x_0 \in \bbR$. For $1>\alpha>0$ let 
\beq \no
\lambda_n=x_0+\frac{\eta}{n^\alpha}
\eeq 
and let $G_0(\lambda_n)=\left(J_0-\lambda_n \right)^{-1}$ and $G(\lambda_n)=\left(J-\lambda_n \right)^{-1}$.
Fix $m$ and assume that there exists $1>\beta>\alpha$ such that 
\beq \label{eq:measure-condition}
\begin{split}
\left(\sum_{j=[n-4mn^\beta]}^{[n+4mn^\beta]}\Im \left( G_0(\lambda_n) \right)_{j,j} \right)=O(n^\beta),
\end{split}
\eeq
and 
\beq \label{eq:jacobi-perturb-cond2}
\left \| \left(J-J_0 \right) P_{\left \{ n-6mn^\beta,n+6mn^\beta \right \} } \right \|_\infty = o \left (n^{-\beta} \right).
\eeq
Then
\beq \label{eq:preconclusion1}
\frac{1}{n^\alpha} \left \|  P_{\left \{ n-2mn^\beta,n+2mn^\beta\right\}} \left( G(\lambda_n)-G_0(\lambda_n)\right) P_{\left \{n-2mn^\beta,n+2mn^\beta\right \}} \right \|_1 \rightarrow 0
\eeq
as $n \rightarrow \infty$.

Alternatively, \eqref{eq:preconclusion1} holds if we replace both \eqref{eq:measure-condition} and \eqref{eq:jacobi-perturb-cond2} with the single condition
\beq \label{eq:jacobi-perturb-cond}
\left \|  \left(J-J_0 \right) P_{\left \{ n-4mn^\beta,n+4mn^\beta \right \} } \right \|_\infty= o \left (n^{-(\beta+\alpha)} \right).
\eeq

\end{proposition}

\begin{proof}
First recall that for any self-adjoint Jacobi matrix, $J$, and $\lambda \in \bbC$ with $\Im\lambda>0$,
\beq \label{eq:resolvent-norm-bound}
\left \|G(J;\lambda) \right \|_\infty \leq \frac{1}{\Im\lambda}.
\eeq
Combining this with Proposition \ref{prop:CT}, we see that both $\calA^{(n)}=\frac{G_0(\lambda_n)}{n^\alpha}$ and $\calB^{(n)}=\frac{G(\lambda_n)}{n^\alpha}$ satisfy Hypothesis \ref{Hypothesis1}.  

For notational simplicity we write 
\begin{align*}
\widetilde{P}_n^1&=P_{\left\{n-2mn^\beta,n+{2mn}^\beta\right\}}, \\
\widetilde{P}_n^2&=P_{\left\{n-4mn^\beta,n+{4mn}^\beta\right\}} \\
\widetilde{P}_n^3&=P_{\left\{n-6mn^\beta,n+{6mn}^\beta\right\}}.
\end{align*}

Now use the resolvent formula 
\begin{equation} \label{eq:resolvent-formula}
G(\lambda_n) -G_0(\lambda_n)=G(\lambda_n)\left(J_0-J \right)G_0(\lambda_n)
\end{equation}
to  write 
\begin{multline} \label{eq:resolvent-consequence}
\widetilde{P}_n^1  \left(G(\lambda_n) -G_0(\lambda_n)\right)\widetilde{P}_n^1  =\widetilde{P}_n^1   G(\lambda_n) \left(J_0-J \right)G_0(\lambda_n) \widetilde{P}_n^1 \\
=\widetilde{P}_n^1G(\lambda_n) \widetilde{P}_n^2 \left(J_0-J \right) \widetilde{P}_n^2 G_0(\lambda_n) \widetilde{P}_n^1 +\mathcal{R}_n.
\end{multline}
By Lemma \ref{lem:C-T-norm-estimate}, the fact that $\frac{G_0(\lambda_n)}{n^\alpha}$ and $\frac{G(\lambda_n)}{n^\alpha}$ satisfy Hypothesis \ref{Hypothesis1}, and by the boundedness of $J_0$ and $J$, we see that for some $D_1,D_2, D_3$ 
\beq \label{eq:remainder-estimate}
\left \| \mathcal{R}_n \right \|_1 \leq D_1 n^{D_2} e^{-D_3 n^{\beta-\alpha}}.
\eeq
Thus we are left with showing
\beq \label{eq:onenormresolvent}
\begin{split}
& \frac{1}{n^\alpha}\left \| \widetilde{P}_n^1G(\lambda_n) \widetilde{P}_n^2 \left(J_0-J \right) \widetilde{P}_n^2 G_0(\lambda_n)\widetilde{P}_n^1 \right \|_1 \rightarrow 0 
\end{split}
\eeq
as $n \rightarrow \infty$.

Let us first assume that we have \eqref{eq:jacobi-perturb-cond}. 
Since $\left \|    G_0(\lambda_n) \right \|_\infty,\left \|    G(\lambda_n) \right \|_\infty \leq Cn^\alpha$ and $\widetilde{P}_n^2$ has rank $4 m n^\beta$, we then get
\begin{multline}
 \left \| \widetilde{P}_n^1 G(\lambda_n)  \widetilde{P}_n^2 \left(J_0-J \right) \widetilde{P}_n^2 G_0(\lambda_n)\widetilde{P}_n^1 \right \|_1 \leq 
C^2 n^{2\alpha}\left  \|  \left(J_0-J \right) \widetilde{P}_n^2\right \|_1\\
 \leq 4C^2m n^{\beta+2\alpha}\left  \| \left(J_0-J \right) \widetilde{P}_n^2\right \|_\infty,
\end{multline}
so that indeed \eqref{eq:onenormresolvent} follows. 

Now let us show that \eqref{eq:onenormresolvent} follows from \eqref{eq:measure-condition} and \eqref{eq:jacobi-perturb-cond2}. We first use the resolvent identity a second time to write
\begin{multline}\label{eq:2ndresolventAA} 
 \widetilde{P}_n^1 G(\lambda_n) \widetilde{P}_n^2 \left(J_0-J \right) \widetilde{P}_n^2\ G_0(\lambda_n)\widetilde{P}_n^1\\=
  \widetilde{P}_n^1 G_0(\lambda_n)(I+\left(J_0-J \right)G(\lambda_n)) \widetilde{P}_n^2 \left(J_0-J \right) \widetilde{P}_n^2 G_0(\lambda_n)\widetilde{P}_n^1
\end{multline}
and so
\begin{multline}\label{eq:2ndresolventA} 
 \left\|\widetilde{P}_n^1 G(\lambda_n) \widetilde{P}_n^2 \left(J_0-J \right) \widetilde{P}_n^2\ G_0(\lambda_n)\widetilde{P}_n^1\right\|_1\\\leq 
 \left\| \widetilde{P}_n^1 G_0(\lambda_n)\right \|_2 \left\|(I+\left(J_0-J \right)G(\lambda_n)) \widetilde{P}_n^2\right\|_\infty \left\| \left(J_0-J \right) \widetilde{P}_n^2\right\|_\infty\left\|  G_0(\lambda_n) \widetilde{P}_n^1\right\|_2.
\end{multline}
As before, we may write
\beq \no \left(J_0-J \right)G(\lambda_n) \widetilde{P}_n^2=\left(J_0-J \right)\widetilde{P}_n^3G(\lambda_n) \widetilde{P}_n^2+ \widetilde{ \mathcal R_n},
\eeq
 where $\widetilde{\mathcal R}_n$  also satisfies \eqref{eq:remainder-estimate} with the $\|\cdot\|_\infty$ norm.  By combining this with the condition \eqref{eq:jacobi-perturb-cond2} and the fact that $\|G(\lambda_n)\|_\infty = \mathcal O(n^\alpha)$, we find 
\begin{equation}  \no \left\|\left(J_0-J \right)G(\lambda_n) \widetilde{P}_n^2\right\|_\infty \leq \left\|\left(J_0-J \right)\widetilde{P}_n^3\right\|_\infty \left\|G(\lambda_n) \widetilde{P}_n^2\right\|_\infty+o(1)=o(1),
\end{equation}
as $n \to \infty$. This implies that
\beq \no
\left\|(I+\left(J_0-J \right)G(\lambda_n)) \widetilde{P}_n^2\right\|_\infty \leq D
\eeq
for some constant $D>0$ which, combined with \eqref{eq:2ndresolventA}, implies that 
\begin{multline}\label{eq:2ndresolventAB} 
 \left\|\widetilde{P}_n^1 G(\lambda_n) \widetilde{P}_n^2 \left(J_0-J \right) \widetilde{P}_n^2\ G_0(\lambda_n)\widetilde{P}_n^1\right\|_1\\=
D \left\| \widetilde{P}_n^1 G_0(\lambda_n)\right\|_2 \left\| \left(J_0-J \right) \widetilde{P}_n^2\right\|_\infty\left\|  G_0(\lambda_n) \widetilde{P}_n^1\right\|_2.
\end{multline} 

To estimate the Hilbert-Schmidt norms, first note that
\beq \no
\begin{split}
2i {\rm Im} \left( G_0(\lambda_n)\right)_{j,j} &=\left( \left(J-\lambda_n \right)^{-1} \right)_{j,j}-\left( \left(J-\overline{\lambda_n} \right)^{-1} \right)_{j,j} \\
&=\left(\left(J-\lambda_n \right)^{-1}\left(\lambda_n-\overline{\lambda_n} \right) \left(J-\overline{\lambda_n} \right)^{-1}\right)_{j,j} \\
&=2i{\rm Im} \lambda_n \left( \left( J_0-\lambda_n \right)^{-1} \left( J_0-\overline{\lambda_n} \right)^{-1}\right)_{j,j}
\end{split}
\eeq
so that
\beq \label{eq:famous-formula}
\Im \left( G_0(\lambda_n) \right)_{j,j}=\Im \lambda_n \sum_{k=1}^\infty \left| \left( G_0(\lambda_n)  \right)_{j,k} \right|^2.
\eeq
It follows that
\begin{multline}
\no
\left \|  \widetilde{P}_n^1 G_0(\lambda_n)  \right \|_2^2= \left \|   G_0(\lambda_n)  \widetilde{P}_n^1 \right \|_2^2 = \sum_{j=[n-2mn^\beta]}^{[n+2mn^\beta]}\sum_{k=1}^\infty  \left| \left( G_0(\lambda_n)  \right)_{k,j} \right|^2\\=\frac{n^\alpha}{\Im\eta}\sum_{j=[n-2mn^\beta]}^{[n+2mn^\beta]}\Im \left( G_0(\lambda_n)  \right)_{j,j}.
\end{multline}
Combining this with  \eqref{eq:2ndresolventAB},  we immediately get
\beq \label{eq:esssss}
\begin{split}
& \left \| \widetilde{P}_n^1 G(\lambda_n)\widetilde{P}_n^2\left(J_0-J \right)\widetilde{P}_n^2G_0(\lambda_n)\widetilde{P}_n^1 \right \|_1 \\
&\quad \leq \frac{D n^\alpha}{\Im\eta} \left\|    \left(J_0-J \right) \widetilde{P}_n^2\right \|_\infty  \left ( \sum_{j=[n-2mn^\beta]}^{[n+2mn^\beta]}\Im \left( G(\lambda_n) \right)_{j,j} \right)
\end{split}
\eeq
which, when combined with \eqref{eq:measure-condition} and \eqref{eq:jacobi-perturb-cond2}, implies \eqref{eq:onenormresolvent}. This ends the proof of the theorem.
\end{proof}

We are now ready to prove Theorem \ref{thm:main-result} for the case $f(x)=\sum_{j=1}^n c_j \Im \frac{1}{\eta_j-x}$.

\begin{theorem} \label{thm:2.1-for-Green}
Let $\mu$ and $\mu_0$ be two measures with finite moments and denote by $\{a_n, b_n\}_{n=1}^\infty$ and $\{a_n^0,b_n^0\}_{n=1}^\infty$ the respective associated recurrence coefficients. Let $x_0 \in \bbR$ be such that there exists a neighborhood $x_0 \in I$ on which the following two conditions hold:\\
(i) $\mu_0$ restricted to $I$ is absolutely continuous with respect to Lebesgue measure and its Radon-Nikodym derivative is bounded there.  \\
(ii) The orthonormal polynomials for $\mu_0$ are uniformly bounded on $I$.\\
Assume further that
$$a_n-a_n^0=\mathcal O(n^{-\beta}), \qquad b_n-b_n^0=\mathcal O(n^{-\beta}),$$
as $n \to \infty$ for some $1>\beta>0$.

Let 
\beq \no
f(x)=\sum_{j=1}^N c_j  \Im \frac{1}{x-\eta_j}
\eeq
where $N \in \bbN$, $c_1, c_2 \ldots c_N \in \bbR$, and $\eta_1, \eta_2,\ldots,\eta_N \in \{\eta \in \bbC \mid \Im \eta>0\}$ are arbitrary.
Then for any $0<\alpha<\beta$ we have
\beq \label{eq:conclusion2}
\left|\bbE \left(X^{(n)}_{f,\alpha,x_0}-\bbE (X^{(n)}_{f,\alpha,x_0} \right)^m -\bbE_0 \left(X^{(n)}_{f,\alpha,x_0}-\bbE_0(X^{(n)}_{f,\alpha,x_0}\right)^m\right| \rightarrow 0
\eeq
as $n \rightarrow \infty$. Here $\bbE$ and $\bbE_0$ denote the expection with respect the orthogonal polynomial ensemble corresponding to $\mu$ and $\mu_0$ respectively.
\end{theorem}

\begin{proof}
First we note that \eqref{eq:conclusion2} is equivalent to 
\beq \no
\left|\calC_m^{(n)} \left(X^{(n)}_{f,\alpha,x_0}\right) -\calC_{m,0}^{(n)} \left(X^{(n)}_{f,\alpha,x_0}\right) \right| \rightarrow 0, \qquad m\geq 2
\eeq
as $n \rightarrow \infty$, where $\calC_{m,0}^{(n)}$ (resp., $\calC_m^{(n)}$) are the cumulants with respect to $\mu_0$ (resp., $\mu$).

Further, letting $J_0$ (resp., $J$) be the Jacobi matrix corresponding to $\mu_0$ (resp., $\mu$), we note that the matrices $f\left(n^\alpha(J_0-x_0) \right)$ and $f\left(n^\alpha(J-x_0) \right)$ satisfy Hypothesis \ref{Hypothesis1} (by Proposition \ref{prop:CT} and since $\min_j \Im \eta_j>0$). Thus, by \eqref{eq:cumulant-spectral-form1} and Theorem \ref{th:comparison-general}, we see that
\beq \no
\begin{split}
& \left|\calC_m^{(n)} \left(X^{(n)}_{f,\alpha,x_0} \right) - \calC_{m,0}^{(n)} \left( X^{(n)}_{f,\alpha,x_0} \right) \right| \\
&\quad \leq \frac{C(m)}{n^\alpha} \left \| P_{\left \{n-2mn^\beta,n+2mn^\beta \right\}} \left( \sum c_j \Im G^{(n,\lambda_j)}-\sum c_j\Im G^{(n,\lambda_j)}_0 \right) P_{\left \{n-2mn^\beta,n+2mn^\beta\right \}} \right \|_1\\
&\qquad+o(1),
\end{split}
\eeq
where $\lambda_j=\lambda_{j,n}=x_0+\frac{\eta_j}{n^\alpha}$.
Therefore, using the triangle inequality, we now see that it suffices to prove that the conditions of the theorem imply those of Proposition \ref{Jacobi-perturbation} for any $m$ and for some $\beta'>\alpha$.

Fix $m$ and $\beta>\beta'>\alpha$. We shall show that \eqref{eq:measure-condition} and \eqref{eq:jacobi-perturb-cond2} hold for $\beta'$. To prove \eqref{eq:jacobi-perturb-cond2} note that by the conditions of the theorem
\beq \no
n^{\beta'} \left \|  \left(J-J_0 \right) P_{\left \{ n-6mn^\beta,n+6mn^\beta \right \} } \right \| \leq C(m) n^{\beta'} \left(n-6mn^{\beta'} \right)^{-\beta} \rightarrow 0
\eeq
as $n \rightarrow \infty$.

To prove \eqref{eq:measure-condition}, let $\{p_j\}_{j=0}^\infty$ be the orthonormal polynomials with respect to $\mu_0$, and recall that 
\beq \label{eq:boundednessgfunction}
\begin{split}
\Im \left( G^{n,\lambda}_0 \right)_{j,j}&= \Im \int \frac{|p_j(x)|^2 \textrm{d} \mu_0(x)}{x-x_0-\frac{\eta}{n^\alpha}}\\
&= \frac{\Im\eta}{n^\alpha}\int \frac{|p_j(x)|^2 \textrm{d} \mu_0(x)}{\left(x-x_0-\frac{\textrm{Re}\eta}{n^\alpha} \right)^2+\frac{\Im\eta^2}{n^{2\alpha}}} \\
&= \frac{\Im\eta}{n^\alpha} \int_{\bbR \setminus I}+\frac{\Im\eta}{n^\alpha}\int_{I}.
\end{split}
\eeq
The first integral in the sum vanishes as $n \rightarrow \infty$, whereas, by the uniform boundedness of the $p_j$ on $I$, the second integral is bounded by a constant times the Poisson transform of $\textrm{d}\mu_0$ restricted to $I$. This, by condition $(i)$ of the theorem, is bounded. This proves \eqref{eq:measure-condition} and ends the proof of the theorem.
\end{proof}

Before we end this section we note the following additional consequence of the arguments above, which will come in handy later on.

\begin{proposition}\label{prop:comparinggreen} Let $\mu_0$ and $\mu$  be as in Theorem \ref{thm:2.1-for-Green}. Then for any fixed $j,k \in \bbZ$  
$$\left(G(\lambda_n)\right)_{n+j,n+k}-\left(G_0(\lambda_n)\right)_{n+j,n+k} \to 0,$$
as $n \rightarrow \infty$. 
\end{proposition}
\begin{proof}
First note that 
$$\left|\left(G(\lambda_n)\right)_{n+j,n+k}-\left(G_0(\lambda_n)\right)_{n+j,n+k} \right|= \left\| \widehat{P}^1_n  (G(\lambda_n)-G_0(\lambda_n)) \bar P_n^1\right\|_1,$$
where $\widehat P^1_n = P_{\left\{n+j-1,n+j-1\right\}}$ and $\bar{ P}_n^1 = P_{\left\{n+k-1,n+k-1\right\}}.$ Now let
\beq \no
\widetilde{P}_n^1=P_{\{n-M,n+M \} }
\eeq
where $M \geq \max(|j|,|k|)$. and note that $\widehat{P}^1_n=\widehat{P}^1_n \widetilde{P}^1_n$ and $\bar P_n^1=\widetilde{P}^1_n \bar P_n^1$. Thus
\beq \no
\left\| \widehat{P}^1_n  (G(\lambda_n)-G_0(\lambda_n)) \bar P_n^1\right\|_1\leq \left\| \widetilde{P}^1_n  (G(\lambda_n)-G_0(\lambda_n)) \widetilde{P}_n^1 \right\|_1.
\eeq

Now, we know from the proof of Theorem \ref{thm:2.1-for-Green} that \eqref{eq:measure-condition} and \eqref{eq:jacobi-perturb-cond2} hold and so we may repeat the proof of Proposition \ref{Jacobi-perturbation} with $\widetilde{P}_n^1$ as defined here and everything else unchanged to see that 
\beq \no
 \left\| \widetilde{P}^1_n  (G(\lambda_n)-G_0(\lambda_n)) \widetilde{P}_n^1 \right\|_1 \leq \frac{D n^\alpha}{\Im\eta} \left\|    \left(J_0-J \right) \widetilde{P}_n^2\right \|_\infty  \left ( \sum_{j=n-M}^{n+M}\Im \left( G(\lambda_n) \right)_{j,j} \right).
\eeq
But \eqref{eq:boundednessgfunction} and the discussion following it show that $\left ( \sum_{j=n-M}^{n+M}\Im \left( G(\lambda_n) \right)_{j,j} \right)$ is bounded and so the right hand side vanishes as $n \rightarrow \infty$, proving the claim.
\end{proof}

%%%%%%%%%%%%%%%%%%%%%%%%%%%%%%%%%%%%%%%%%%%%%%%%%%%%%%%%%

\section{The free Jacobi operator and the proof of Theorem \ref{thm:CLTgeneral} in case $f(x)= \sum_{j=1}^N c_j \Im \frac{1}{x-\eta_j}$} \label{sec:CLTfree}

In order to deduce a general Central Limit Theorem from Theorem \ref{th:comparison-general} we need a model for which a Central Limit can be proved directly.  We therefore consider in this section the special case of $\textrm{d}\mu_0(x) =\frac{\sqrt{4-x^2}}{2 \pi}\chi_{[-2,2]}(x)\textrm{d}x$. The associated Jacobi operator in this case is known as the free Jacobi operator and has the simple form $$a_n \equiv 1 \text{   and   } b_n \equiv 0.$$  By exploiting the fact that the free Jacobi operator is also a Toeplitz operator, we will prove Theorem \ref{thm:CLTgeneral} for  functions $f$ of the form  (cf. \eqref{eq:specialclass}) 
\beq \label{eq:defffree}
f(x)=\sum_{j=1}^N c_j  \Im \frac{1}{x-\eta_j}
\eeq
where $N \in \bbN$, $c_1, c_2 \ldots c_N \in \bbR$, and $\eta_1, \eta_2,\ldots,\eta_N \in \{\eta \in \bbC \mid \Im \eta>0\}$ are arbitrary. At the end of the section we then combine this Central Limit Theorem  for the free Jacobi operator with Theorem \ref{th:comparison-general} and obtain Theorem \ref{thm:CLTgeneral} for such functions.

\subsection{Reduction to a Toeplitz operator}
Recall (\eqref{eq:defgreen}, \eqref{eq:defImaggreen} and \eqref{eq:lincomImaggreen}) we use $G(\lambda)$ for the resolvent, $ F(\lambda)$ for the imaginary part of the resolvent and $F(\vec{\lambda},\vec{c})= \sum_{j=1}^N c_j F(\lambda_j)$ for a general linear combination of the imaginary part of the resolvent at various points. 

In the free case, it is well-known (and easy to verify) that 
\begin{equation}\label{eq:greendsfree}
(G(\lambda))_{j,k}= \frac{\omega(\lambda)^{|j-k|}-\omega(\lambda)^{j+k}}{\omega(\lambda)-1/\omega(\lambda)}, \qquad j,k=1,2,\dots
\end{equation}
where 
$$\omega(\lambda)= \frac{\lambda-\sqrt{\lambda^2-4}}{2},$$
where the square root is taken such that $\lambda \to \omega(\lambda)$ is analytic in $\bbC \setminus[-2,2]$ and $\omega(\lambda)= \mathcal O(1/\lambda)$ as $\lambda \to \infty$. Note that this map maps  $\bbC \setminus[-2,2]$ onto  the open unit disk and hence $|\omega(\lambda)|<1$. 

Given a Laurent series $a(z)= \sum_{k} a_k z^k$, recall that the associated Toeplitz matrix $T(a)$ and Hankel matrix $ H(a)$ are defined as the one-sided infinite matrices
$$\left(T(a)\right)_{j,k}=a_{j-k}, \quad \text{and} \quad  \left(H(a)\right)_{j,k}=a_{j+k-1}, \qquad j,k=1,2,3,\ldots$$  From \eqref{eq:greendsfree} we thus see that  the resolvent $G(\lambda)$ for the free Jacobi operator can be written as a  sum of a Toeplitz and a Hankel operator, and hence also 
$$F(\lambda)= T(a_\lambda)+ H(b_\lambda)$$
where 
\begin{multline}\label{eq:defal} a_\lambda(z) =\sum_{j\in \bbZ} \left( \Im \frac{\omega(\lambda)^{|j|}}{\omega(\lambda)-\omega(\lambda)^{-1}}   \right) z^j \\
=\frac{1}{2 {\rm i}} \left(\frac{z}{(z-\omega(\lambda))(z-1/\omega(\lambda))}-\frac{z}{(z-\omega\left(\overline{\lambda} \right))\left(z-1/\omega\left(\overline{\lambda} \right) \right)}\right),\end{multline}
and 
\begin{multline*}
b_\lambda(z) =-z^{-1} a_\lambda(z)\\=-\frac{1}{2 {\rm i}} \left(\frac{ 1}{\left(z-\omega(\lambda) \right) \left(z-1/\omega(\lambda) \right)}-\frac{1}{\left(z-\omega\left(\overline{\lambda} \right) \right)\left(z-1/\omega\left(\overline{\lambda} \right) \right)}\right).
\end{multline*}
For $N\in \bbN, \eta_1,\ldots,\eta_N \in \bbH_+= \{ \eta \in \bbC \mid \Im \eta>0\}$ and $c_1,\ldots,c_N \in \bbR$ we set
$$\lambda_j=\lambda_j(n)= x_0+ \frac{\eta_j}{n^\alpha}$$ 
and define 
\begin{equation}\label{eq:defphifree}
\phi^{(n)}= \frac{1}{n^\alpha} \sum_{j=1}^N c_j a_{\lambda_j}.
\end{equation}
Note that 
\begin{equation}\label{eq:Ftoeplitzplushankel} 
\frac{1}{n^\alpha} F(\vec \lambda,\vec c) = T(\phi^{(n)})+H(-z^{-1}\phi^{(n)}).
\end{equation}

Clearly, Proposition \ref{prop:CT}, together with the fact that $\|G(\lambda) \|_\infty \leq \frac{1}{\textrm{Im}\lambda}$, implies that  both $\frac{1}{n^\alpha}G\left(\vec \lambda, \vec c \right)$ and $\frac{1}{n^\alpha}F\left(\vec \lambda, \vec c \right)$ satisfy Hypothesis \ref{Hypothesis1} as long as $\eta_1, \eta_2,\ldots,\eta_N \in$ a compact subset of $\bbH_+$. 
In fact, the following lemma shows that in the free case a substantially stronger bound is true by giving a useful bound on the coefficients of $\phi^{(n)}$. 
\begin{lemma}
Let $x_0 \in (-2,2)$, $0<\alpha<1$ and fix a compact $K\subset \bbH_+=\{\eta \in \bbC \mid \Im \eta>0\}$. Write $\phi^{(n)}= \sum_{k\in \bbZ} \phi^{(n)}_k z^k$.  Then there exist $d_1,d_2>0$ such that 
\begin{equation}\label{eq:freeCT}
|\phi^{(n)}_k| \leq  \frac{d_1}{n^\alpha} \left(\sum_{j=1}^N|c_j|\right) \exp(-d_2 |k| /n^\alpha),\end{equation}
for $k\in \bbZ,n,N\in \bbN$,  $\eta_1,\ldots,\eta_N \in K$ and $c_1,\ldots,c_N \in  \bbR$. In particular, $n^{-\alpha} G$ and $ n^{-\alpha} F$  satisfy Hypothesis \ref{Hypothesis1}.
\end{lemma}
\begin{proof}
By the triangle inequality, it is sufficient to deal with the case $N=1$. Hence we consider $\phi^{(n)}=c a_{\lambda(n)}$.

To prove \eqref{eq:freeCT} we first note that,  
\begin{equation}\label{eq:asymptoticsomega}\omega(\lambda)= \begin{cases}
\frac{x_0 - {\rm i}  \sqrt{4-x_0^2}}{2}\left(1+ \frac{{\rm i } \eta }{n^{\alpha} \sqrt{4-x_0^2} } + \mathcal O(n^{-2\alpha})\right),& \Im \eta>0\\
\frac{x_0+ {\rm i}  \sqrt{4-x_0^2}}{2}\left(1- \frac{{\rm i } \eta }{n^{\alpha} \sqrt{4-x_0^2} } + \mathcal O(n^{-2\alpha})\right), & \Im \eta<0
\end{cases} 
\end{equation}
as $n \to \infty$, uniformly for $\eta \in K$. Thus there exist constants $d_1,d_2>0$  such that 
 such that 
$$|\omega(\lambda)|\leq 1-\frac{d_1}{n^\alpha} \leq \exp(-d_2/n^\alpha), $$
and $1/|\omega(\lambda)-1/\omega(\lambda)|\leq d_1$. 
By inserting these inequalities into  \eqref{eq:defal}, we obtain \eqref{eq:freeCT}. 
\end{proof}

Using the notation \eqref{eq:newcumu}, we see from Theorem \ref{th:comparison-general} 
that for $m \in \bbN$ we have
\beq \label{eq:HankelThrow}
\begin{split}
& \left|C_m^{(n)} \left (\frac{1}{n^\alpha} F(\vec \lambda, \vec c) \right)-C_m^{(n)} \left( T(\phi^{(n)}) \right) \right| \\ 
&\quad \leq C(m) \left \| P_{\left \{ n-2mn^\beta,n+2mn^\beta \right \}} H(z\phi^{(n)}) P_{\left \{ n-2mn^\beta,n+2mn^\beta \right \}} \right \|_1+o(1) \\
&\quad \leq  \widetilde{C}(m) n^{2\beta} e^{d_2 n^{\beta-\alpha}}e^{-d_2 n^{1-\alpha}}=o(1)
\end{split}
\eeq
as $n \to \infty$, by \eqref{eq:Ftoeplitzplushankel} and \eqref{eq:freeCT}.
 
Our problem is thus reduced to computing the limits of $C_m^{(n)} \left(T \left(\phi^{(n)} \right) \right)$. Repeating the calculation at the beginning of Section \ref{sec:DetStruct} and keeping in mind that $C_1^{(n)} \left(T \left(\phi^{(n)} \right) \right) =\mathrm{Tr} P_n T \left(\phi^{(n)}\right)$ we see that 
\begin{multline}\label{eq:cumultoepl}
\exp\left(\sum_{m=2}^\infty t^m C_m^{(n)} \left(T \left(\phi^{(n)}\right)\right)  \right) \\
 =\det\left(I+ P_n\left({\rm e}^{ t T\left(\phi^{(n)}\right)}-I \right)P_n\right)  {\rm e}^{-t \Tr P_n T \left(\phi^{(n)}\right)}.
 \end{multline}
We will proceed by analyzing the asymptotic behavior of the Fredholm determinant at the right-hand side.

\subsection{A Fredholm determinant identity}

The first step in analyzing the right-hand side of \eqref{eq:cumultoepl} is a rewriting of the determinant. The arguments given in this subsection are based on ideas that lie behind the Strong Szeg\H{o} Limit Theorem for Toeplitz determinants.

Let us consider a general symbol $\phi(z)= \sum_j \phi_j z^j$. Split $T(\phi)$ as $$T(\phi)= T(\phi_{+}) + T(\phi_{-}),$$
where 
$\phi_{+}=\sum_{j\geq 0 } \phi_j  z^j $  and $\phi_{-}= \sum_{j <0 } \phi_j z^j.$ 
\begin{lemma}
We have that 
 \begin{multline}\label{eq:Toeplitztickdet}
 \det\left(I+ P_n({\rm e}^{ t T(\phi)}-I)P_n\right)  {\rm e}^{-\mathrm{Tr} P_nT(\phi)}\\=\det\left(I+ P_n \left({\rm e}^{-t T(\phi_{+} ) }{\rm e}^{ t T(\phi)}{\rm e}^{-t T(\phi_{-} ) }-I \right)\right) . 
 \end{multline}
 \end{lemma}
 \begin{proof}
  The idea is standard and in this form has been used in, for example, \cite{BW,BD-CLT}. For completeness we include a proof.
  
  Since $T(\phi_+)$ is lower triangular we have $P_n T(\phi_+) P_n = P_n T(\phi_+)$ and hence also 
  $$P_n {\rm e}^{-t  T(\phi_+) }P_n= P_n{\rm e}^{-t  T(\phi_+) },$$
  and
  $${\rm e}^{-t P_n T(\phi_+) P_n}=(I-P_n) + P_n{\rm e}^{- tT(\phi_+)}P_n.$$
  In the same way, since $T(\phi_{-})$ is upper triangular we have 
   $$P_n {\rm e}^{-t  T(\phi_-) }P_n= {\rm e}^{-t  T(\phi_-) }P_n,$$
   and 
$$  {\rm e}^{-t P_n T(\phi_-) P_n}=(I-P_n) +P_n {\rm e}^{- t T(\phi_-)}P_n.$$
From these identities we find 
\begin{multline*}
{\rm e}^{- t P_n T(\phi_+) P_n}\left(I+ P_n({\rm e}^{ t T(\phi)}-I)P_n\right)  {\rm e}^{- t P_n T(\phi_-) P_n}\\
=\left(I-P_n+ P_n {\rm e}^{- t T(\phi_+) } P_n\right)\left(I-P_n+  P_n{\rm e}^{ t T(\phi)} P_n\right)  \left(I-P_n+ P_n {\rm e}^{- t T(\phi_-)P_n}\right)
\\
= I-P_n +  P_n {\rm e}^{- t T(\phi_+) } P_n{\rm e}^{ t T(\phi)} P_n {\rm e}^{- t T(\phi_-)} P_n \\
=  I-P_n +  P_n {\rm e}^{- t T(\phi_+)} {\rm e}^{ t T(\phi)}  {\rm e}^{- t T(\phi_-)} P_n\\
=I+ P_n \left({\rm e}^{- T(\phi_+) }   {\rm e}^{ T(\phi) } {\rm e}^{- T(\phi_-) }-I\right)P_n.
\end{multline*}
Therefore,
\begin{multline}\label{eq:Toepltiztickdethulp}
\det \left({\rm e}^{- t P_n T(\phi_+) P_n}\left(I+ P_n({\rm e}^{ t T(\phi)}-I)P_n\right)  {\rm e}^{- t P_n T(\phi_-) P_n} \right)\\
= \det \left(I+ P_n \left({\rm e}^{- T(\phi_+) }   {\rm e}^{ T(\phi) } {\rm e}^{- T(\phi_-) }-I\right)P_n\right).
\end{multline}
By factorizing the the determinant on the left-hand side we find
\begin{multline*}
\det \left({\rm e}^{- t P_n T(\phi_+) P_n}\left(I+ P_n({\rm e}^{ t T(\phi)}-I)P_n\right)  {\rm e}^{- t P_n T(\phi_-) P_n} \right)
\\
=\det {\rm e}^{- t P_n T(\phi_+) P_n}\det \left(I+ P_n({\rm e}^{ t T(\phi)}-I)P_n\right)  \det {\rm e}^{- t P_n T(\phi_-) P_n}\\
=\det \left(I+ P_n({\rm e}^{ t T(\phi)}-I)P_n\right)  {\rm e}^{- \Tr t P_n T(\phi)}.
\end{multline*}
By inserting the latter into \eqref{eq:Toepltiztickdethulp} we obtain \eqref{eq:Toeplitztickdet}, proving the statement.
 \end{proof}

To see how \eqref{eq:Toeplitztickdet} brings us closer to proving a Central Limit Theorem, replace $P_n$ by $I$ on its right-hand side. We then obtain the determinant of the product of three exponentials, which can be computed using the following principle: Suppose $A,B$ are two bounded operators  such that $[A,B]$ is trace class. Then 
$${\rm e}^{-A} {\rm e}^{A+B} {\rm e}^{-B}-I \text{ is trace class}, $$ and 
\begin{equation}\label{eq:erhardt}
\det {\rm e}^{-A} {\rm e}^{A+B} {\rm e}^{-B} = {\rm e}^{-\frac{1}{2} \Tr [A,B]}.
\end{equation}
This beautiful identity was first proved by Ehrhardt \cite{E} and is a generalization of the Helton-Howe-Pincus formula. By considering $A=tT\left (\phi_+ \right)$, $B=t T\left(\phi_- \right)$, and under the condition that $[T(\phi_+),T(\phi_-)]$ is trace class, the right-hand side of \eqref{eq:erhardt} is immediately seen to be the generating function for a Gaussian random variable. For fixed $\phi$, as in \cite{BD-CLT}, the proof of a Central Limit Theorem in the case of Toeplitz matrices follows from this almost immediately. In our case, however, $\phi$ depends on $n$. This requires a somewhat more subtle reasoning, and in particular, control of the trace norm of the remainder term. The following formula  for ${\rm e}^{-A} {\rm e}^{A+B} {\rm e}^{-B}-I $ in terms of the commutator $[A,B]$ will allow us to do this, also showing along the way that ${\rm e}^{-A} {\rm e}^{A+B} {\rm e}^{-B}-I $ is indeed trace class.

\begin{lemma}
For any bounded operators $A,B$ we have
\begin{multline}\label{eq:expansionAB}
{\rm e}^{-A} {\rm e}^{A+B} {\rm e}^{-B}-I\\
=\sum_{m_1,m_2,m_3=0}^\infty \sum_{j=0}^{m_2-1} \frac{(-1)^{m_1+m_3} A^{m_1} (A+B)^j [A,B] (A+B)^{m_2-j-1} B^{m_3}}{m_1! m_2! m_3! (m_1+m_2+m_3+1)} .
\end{multline}
Moroever, if $[A,B]$ is trace class then  ${\rm e}^{-A} {\rm e}^{A+B} {\rm e}^{-B}-I$ is trace class.
\end{lemma}
\begin{proof}
A direct computation shows that 
$$
\frac{{\rm d}}{{\rm d}t} {\rm e}^{- t A} {\rm e}^{t(A+B)} {\rm e}^{-tB}= {\rm e}^{-t A} [B,{\rm e}^{t(A+B)}] {\rm e}^{-t B},
$$
and by expanding the right-hand side we find 
$$\frac{{\rm d}}{{\rm d}t} {\rm e}^{- t A} {\rm e}^{t(A+B)} {\rm e}^{-tB}= \sum_{m_1,m_2,m_3=0}^\infty  \frac{t^{m_1+m_2+m_3}(-1)^{m_1+m_3}}{m_1! m_2! m_3! } A^{m_1}  [A,(A+B)^{m_2}]  B^{m_3}.$$
By inserting the telescoping sum $[c,d^m]= \sum_{j=0}^{m-1} d^j[c,d] d^{m-1-j}$ and using $[A,A+B]= [A,B]$ we find 
\begin{multline*}\frac{{\rm d}}{{\rm d}t} {\rm e}^{- t A} {\rm e}^{t(A+B)} {\rm e}^{-tB}\\=\sum_{m_1,m_2,m_3=0}^\infty \sum_{j=0}^{m_2-1} \frac{t^{m_1+m_2+m_3}(-1)^{m_1+m_3}A^{m_1} (A+B)^j [A,B] (A+B)^{m_2-j-1} B^{m_3}}{m_1! m_2! m_3!} .
\end{multline*}
By integrating the left- and right-hand side from $t=0$ to $t=1$ we obtain the statement.
\end{proof}

We want to use the above results, \eqref{eq:erhardt} in particular, with $A=tT(\phi_{+})$ and $B=tT(\phi_{-})$.  By a direct computation, or by using the general formula for Toeplitz matrices \begin{equation} \label{eq:productToeplitz}
T(ab)= T(a) T(b) + H(a) H\left(\widetilde b \right)\end{equation}
 where $\widetilde b(z)=b(1/z)$, one easily computes that 
\begin{equation}\label{eq:commutatortoeplitz}
 [T(\phi_{+}),T(\phi_{-}) ]= -H(\phi) H \left(\widetilde \phi \right).
 \end{equation}
Moreover, if $ \sum_{k \in \bbZ}  |k|  \phi_k^2<\infty$ then the commutator is trace class since
\begin{multline} \label{eq:tracenormcommutator}
\|[T(\phi_{+}),T(\phi_{-}) ]\|_1= \left\|H(\phi) H\left(\widetilde \phi \right)\right\|_1\leq \left\|H(\phi)\|_2^{1/2} H\left(\widetilde \phi \right) \right\|_2^{1/2}\\= \left( \sum_{k=1}^\infty k  |\phi_k|^2\right)^{1/2}\left( \sum_{k=1}^\infty k  |\phi_{-k}|^2\right)^{1/2}<\infty.
\end{multline}
\begin{lemma}
If  $ \sum_{k \in \bbZ} |k|  \phi_k^2<\infty$ we have that 
 \begin{multline}\label{eq:borodinokounkov}
 \det\left(I+ P_n({\rm e}^{ t T(\phi)}-I)P_n\right)  {\rm e}^{-\mathrm{Tr} P_ntT(\phi)}\\={\rm e}^{\frac{t^2}{2}\Tr H(\phi)H(\tilde \phi)}\det(I+Q_n (R(t,\phi)^{-1}-I)),
 \end{multline}
 where $$R(t,\phi)={\rm e}^{-t T(\phi_{+} ) }{\rm e}^{ t T(\phi)}{\rm e}^{-t T(\phi_{-} ) }.$$
 \end{lemma}
 \begin{proof}
 From \eqref{eq:Toeplitztickdet}
 \begin{multline}
 \det\left(I+ P_n({\rm e}^{ t T(\phi)}-I)P_n\right)  {\rm e}^{-\mathrm{Tr} P_ntT(\phi)}=
\det(I+ P_n (R(t,\phi)-I) )\\= \det(R(t)+Q_n (I-R(t,\phi)))= \det R(t,\phi)\det(I+Q_n (R_n(t,\phi)^{-1}-I)).
\end{multline}
Now  use \eqref{eq:erhardt} and \eqref{eq:commutatortoeplitz} to deduce
$$
\det R(t,\phi)= {\rm e}^{-\frac{t^2}{2} \Tr  [T(\phi_{+}),T(\phi_{-} )]}={\rm e}^{\frac{t^2}{2}\Tr  H(\phi)H\left(\widetilde \phi \right)}.
$$
This proves the statement.
 \end{proof}
\begin{remark}
The identity \eqref{eq:borodinokounkov} is similar to the celebrated identity for Toeplitz determinants as found by Case and Geronimo \cite{cg}, and rediscoved by Borodin and Okounkov \cite{bo}. In fact, the second proof of the latter identity in \cite{BW} by Basor and Widom  was an inspiration for the proof that we present here (see also \cite{BD-CLT}). 
\end{remark}

\subsection{Asymptotic behavior of the Toeplitz determinant}
We return now to the determinant on the right-hand side of \eqref{eq:cumultoepl} and consider $\phi=\phi^{(n)}$ as defined in \eqref{eq:defphifree}. To apply the results of the previous subsection, and in particular \eqref{eq:borodinokounkov}, we need to check that $\sum_{k \in \bbZ} |k| \left|\phi_k^{(n)} \right|^2<\infty$. To this end, we note that it follows from \eqref{eq:freeCT} that for $n\in \bbN$ we have 
\begin{equation}\label{eq:unifromnormsobolevphi}
 \sum_{k \in \bbZ} |k| \left|\phi^{(n)}_k \right|^2 \leq C
 \end{equation}
for some constant $C>0$ that is independent of $n$.  

We divide the analysis of the right-hand side of \eqref{eq:borodinokounkov} into two parts. We will compute the limiting behavior of the trace in the exponential, and show that the determinant converges to $1$. We start with the latter.

We recall that Fredholm determinants are continuous with respect to the trace norm, which means that to prove that $\det(I+A_n) \to 1$ for some sequence of operators $A_n$ it suffices to prove that $\|A_n\|_1 \to 0$. 
\begin{lemma}\label{lem:smallnormproof}
Fix $x_0 \in (-2,2)$ and $0<\alpha<1$. Let $N \in \bbN$, $\eta_1,\ldots, \eta_N \in \bbH_+$, $c_1, \ldots,c_N \in \bbR$  and set $\lambda_j= x_0+\eta_j/n^\alpha$. Then with $\phi^{(n)}$ as in \eqref{eq:defphifree} we have 
$$\left\|Q_n \left(R\left( t,\phi^{(n)} \right)^{-1}-I \right) \right\|_1\to 0,$$
as $n \to \infty$, uniformly for $\eta_j$ in compact subsets of $\bbH_+$ and $t$ in sufficiently small neighborhoods of the origin. 
\end{lemma}
\begin{proof}
 Note that for fixed $\phi$ independent of $n$, the lemma would be a triviality, since $R\left(t,\phi \right)^{-1}-I$ is trace class and $Q_n\to 0$ weakly. However, since $\phi^{(n)}$ depends on $n$  this is not a priori obvious and thus requires a proof.

We use the expansion \eqref{eq:expansionAB} with $A=-tT\left(\phi_+^{(n)} \right)$ and $B= -tT\left(\phi_-^{(n)}\right)$.
\begin{multline}
Q_n   \left(R( t,\phi^{(n)})^{-1}-I\right)=Q_n\left({\rm e}^{-t T\left(\phi_+^{(n)}\right ) }{\rm e}^{t T\left(\phi^{(n)}\right))}{\rm e}^{-t T\left(\phi^{(n)}_{-}\right)}-I\right)\\=
\sum_{m_1,m_2,m_3=0}^\infty \sum_{j=0}^{m_2-1} \frac{(-1)^{m_1+m_3} t^{m_1+m_2+m_3+1}  }{m_1! m_2! m_3! (m_1+m_2+m_3+1)} \\
\times Q_n T\left(\phi_+^{(n)}\right)^{m_1}T\left(\phi^{(n)}\right)^j \left[T\left(\phi_+^{(n)}\right),T\left(\phi_-^{(n)}\right)\right] T\left(\phi^{(n)}\right)^{m_2-j-1} T\left(\phi_-^{(n)}\right)^{m_3}.
\end{multline} By \eqref{eq:freeCT} we see that for each compact set $K\subset \bbH_+$ there exists a constant $C>0$ such that 
$$\left\|T\left(\phi_{\pm}^{(n)}\right) \right\|_\infty  =\left \| \phi_{\pm}^{(n)}\right\|_\infty \leq \sum_{ j= 0 }^{\pm \infty}\left | \phi_{j}^{(n)}\right| \leq C, \qquad \left \|T\left(\phi^{(n)}\right) \right\|_\infty   \leq C.$$
By \eqref{eq:tracenormcommutator} and \eqref{eq:unifromnormsobolevphi}, we find
$$\left\|\left[T\left(\phi_+^{(n)}\right),T\left(\phi_-^{(n)}\right) \right]\right\|_1\leq \left\|H\left(\phi^{(n)} \right) \right\|_2\left\|H\left(\widetilde \phi^{(n)}\right)\right\|_2 \leq C$$
for $n \in \bbN$ and $\eta_j \in K$. Hence, by using together with $\left\|AB \right\|_1\leq \left\|A \right\|_1 \left\|B \right\|_\infty,\left\|AB \right\|_1\leq \left\|A \right\|_\infty \left\|B \right\|_1$  and $\left\|Q_n \right\|_\infty=1$,we find
\begin{multline*}\left\|Q_n T\left(\phi_+^{(n)} \right)^{m_1}T\left(\phi^{(n)} \right)^j\left [T\left(\phi_+^{(n)} \right),T\left(\phi_-^{(n)}\right) \right] T\left(\phi^{(n)}\right)^{m_2-j-1} T\left(\phi_-^{(n)}\right)^{m_3} \right \|_1\\ \leq
\left\|T\left(\phi_-^{(n)} \right) \right\|_\infty^{m_1}\left \|T(\phi^{(n)})] \right\|^{m_2-1}_\infty \left\|\left[T\left(\phi_+^{(n)}\right) ,T\left(\phi_{-}^{(n)}\right)\right] \right\|_1 \left\|T\left(\phi_-^{(n)}\right) \right\|_\infty^{m_3}
\\ \leq  C^{m_1+m_2+m_3},\end{multline*}
for $n \in \bbN$ and $\eta_j \in K$.  Therefore we see that, by dominated convergence, it is sufficient to prove that 
\begin{equation}\label{eq:estnormQn}
\left\|Q_n T\left(\phi_+^{(n)}\right)^{m_1}T\left(\phi^{(n)}\right)^j \left [T\left(\phi_+^{(n)}\right),T\left(\phi_-^{(n)}\right)\right] T\left(\phi^{(n)}\right)^{m_2-j-1} T\left(\phi_-^{(n)}\right)^{m_3}\right\|_1\to 0,
\end{equation}
 for each $m_1,m_2,m_3 \in \bbN$ and $j \in \{0,1,\ldots,m_2-1\}$. 

To prove \eqref{eq:estnormQn} we exploit the fact that the Toeplitz matrices are essentially banded, in the sense that they satisfy Hypothesis \ref{Hypothesis1}. Hence, if we take $1>\beta >\alpha$ and  $M \geq n^\beta$, and then split
 $$Q_M T\left(\phi^{(n)}\right) = Q_M T\left(\phi^{(n)}\right) Q_{M-n^\beta} +Q_M T\left(\phi^{(n)}\right) P_{M-n^\beta} ,$$
then the norm of $Q_M T\left(\phi^{(n)}\right) P_{M-n^\beta}$ can be estimated to be 
$$\left\|Q_M T\left(\phi^{(n)}\right) P_{M-n^\beta}\right\|_\infty  \leq \sum_{j=n^\beta}^\infty\left | \left(\phi^{(n)} \right)_j \right| \leq \exp \left(-D n^{\beta-\alpha}\right),$$
for some constant $D>0$ independent of $M$. Clearly, since $Q_M T\left(\phi^{(n)}\right) P_{M-n^\beta}= Q_M T\left(\phi_+^{(n)}\right) P_{M-n^\beta}$ the same estimate holds for $Q_M T\left(\phi_+^{(n)}\right) P_{M-n^\beta}$. Hence, by combining this esimate with the straightforward estimate from above, we obtain
\begin{multline*}\left\|Q_n T\left(\phi_+^{(n)} \right)^{m_1}T\left(\phi^{(n)}\right)^j \left[T\left(\phi_+^{(n)}\right),T\left(\phi_-^{(n)}\right)\right] T\left(\phi^{(n)}\right)^{m_2-j-1} T\left(\phi_-^{(n)}\right)^{m_3}\right\|_1\\
 \leq    C \left\|Q_{n-n^\beta} T\left(\phi_+^{(n)}\right)^{m_1-1}T\left(\phi^{(n)}\right)^j\left [T\left(\phi_+^{(n)}\right),T\left(\phi_-^{(n)}\right)\right] T\left(\phi^{(n)}\right)^{m_2-j-1} T\left(\phi_-^{(n)}\right)^{m_3}\right\|_1\\+\mathcal O\left(\exp \left(-D n^{\beta-\alpha}\right)\right),\end{multline*}
as $n \to \infty$.  By iterating this procedure we obtain 
\begin{multline*} \left\|Q_n T\left(\phi_+^{(n)}\right)^{m_1}T\left(\phi^{(n)}\right)^j\left [T\left(\phi_+^{(n)}\right),T\left(\phi_-^{(n)}\right)\right] T\left(\phi^{(n)}\right)^{m_2-j-1} T\left(\phi_-^{(n)}\right)^{m_3}\right\|_1
\\
\leq  C^{m_1+j} \left\|Q_{n-(m_1+j)n^{\beta}}\left [T\left(\phi_+^{(n)}\right),T\left(\phi_-^{(n)}\right)\right] T\left(\phi^{(n)}\right)^{m_2-j-1} T\left(\phi_-^{(n)}\right)^{m_3}\right\|_1\\+ \mathcal O(\exp \left(-D n^{\beta-\alpha}\right)),\end{multline*}
as $n \to \infty$.
Hence, to prove \eqref{eq:estnormQn} it is sufficient to prove
$$\left \|Q_{n-(m_1+j)n^{\beta}}\left [T\left(\phi_+^{(n)}\right),T\left(\phi_-^{(n)}\right)\right] T\left(\phi^{(n)}\right)^{m_2-j-1} T\left(\phi_-^{(n)}\right)^{m_3}\right\|_1\to 0.$$
By estimating the trace norm as before we are left with proving 
\beq \label{eq:Qnm1j}\left \|Q_{n-(m_1+j)n^{\beta}}\left [T\left(\phi_+^{(n)}\right),T\left(\phi_-^{(n)}\right)\right]\right\|_1 \to 0,
\end{equation}
as $n \to \infty$. But also the latter is exponentially small as $n \to \infty$. Indeed, by  $\left[T\left(\phi_+^{(n)}\right),T\left(\phi_-^{(n)}\right)\right]=-H\left(\phi^{(n)}\right)H\left(\widetilde \phi^{(n)}\right)$ we have
$$\left\|Q_{n-(m_1+j)n^{\beta}}\left [T\left(\phi_+^{(n)}\right),T\left(\phi_-^{(n)}\right)\right]\right\|_1 \leq \left \|Q_{n-(m_1+j)n^{\beta}} H\left(\phi^{(n)}\right)\right\|_2\left\|H\left(\widetilde \phi^{(n)}\right) \right\|_2.$$
Now, $ \left\|H\left(\widetilde \phi^{(n)}\right) \right\|_2$ is again uniformly bounded and 
$$\left\|Q_{n-(m_1+j)n^{\beta}} H(\phi^{(n)}) \right\|_2^2= \sum_{\ell=1}^\infty \ell \left |\phi^{(n)}_{n-(m_1+j)n^\beta+ \ell} \right|^2,$$
which, by \eqref{eq:freeCT}  is exponentially small as $n \to \infty$.  We therefore obtain \eqref{eq:Qnm1j} and the statement follows.
\end{proof}

It remains to compute the limiting behavior of the trace of the commutator in \eqref{eq:borodinokounkov}, which we do in the next lemma.
\begin{lemma}\label{lem:variancefree}
Fix $x_0 \in (-2,2)$ and $0<\alpha<1$. Let $N \in \bbN$, $\eta_1,\ldots, \eta_N \in \bbH_+$, $c_1, \ldots,c_N \in \bbR$  and set $\lambda_j= x_0+\eta_j/n^\alpha$. Then with $\phi^{(n)}$ as in \eqref{eq:defphifree} we have 
\begin{equation}
\label{eq:limvariancefree}\lim_{n\to \infty} \Tr  H\left(\phi^{(n)}\right)H\left(\widetilde \phi^{(n)} \right)= -\frac{1}{2}\sum_{1\leq i,j\leq N} c_ic_j \Re \frac{1 }{(\eta_i-\bar \eta_j)^2}.
\end{equation}
Moreover, the convergence is uniform for $c_j$ in compact subsets of $\bbR$ and $\eta_j$ in compact subsets of the upper half plane. 
\end{lemma}
\begin{proof}
We use the notation $$g_{\lambda}(z)=\frac{1}{n^\alpha}\frac{1}{\omega(\lambda)-\frac{1}{\omega(\lambda)}} \sum_{j\in \bbZ} \omega(\lambda)^{|j|} z^j=\frac{z}{(z-\omega(\lambda))(z-1/\omega(\lambda))},$$
and hence $\frac{1}{n^\alpha}a_{\lambda}=\frac{1}{2i}\left( g_{\lambda}-g_{\bar \lambda}\right)$.
By a direct verification or using \eqref{eq:productToeplitz} we find 
$$ \Tr  H(g_{\lambda_1}) H (\widetilde{g}_{\lambda_2})= \sum_{j=1}^\infty k (g_{\lambda_1})_k (g_{\lambda_2})_{-k}.$$
The latter sum can be explicitly computed and we obtain 
$$ \Tr  H(g_{\lambda_1}) H (\tilde{g}_{{\lambda_2}})=\frac{1}{n^{2 \alpha}}\frac{\omega(\lambda_1)\omega(\lambda_2)}{(1-\omega(\lambda_1)\omega(\lambda_2))^2} \frac{1}{\frac{1}{\omega(\lambda_1)}-\omega(\lambda_1)} \frac{1}{\frac{1}{\omega(\lambda_2)}-\omega(\lambda_2)}. $$
Now if $\lambda=x_0+ \eta/n^\alpha$ then 
$$\omega(\lambda)= 
\begin{cases}
\frac{x_0-{\rm i} \sqrt{4-x_0^2}}{2} \left(1+\frac{{\rm i} \eta }{\sqrt{4-x_0^2}n^\alpha} + \mathcal O(n^{-2\alpha})\right),& \Im \eta>0,\\
\frac{x_0+{\rm i} \sqrt{4-x_0^2}}{2} \left(1-\frac{{\rm i} \eta}{\sqrt{4-x_0^2}n^\alpha} + \mathcal O(n^{-2\alpha})\right),& \Im \eta<0,\end{cases}
$$
as $n\to \infty$, uniformly for $\eta$ in compact subsets of $\bbH_+$. Hence 
$$\lim_{n\to \infty} \Tr  H(g_{\lambda_1}) H (\tilde{g}_{{\lambda_2}})=\begin{cases}
0, &\text{ if } \Im \eta_1 \cdot \Im \eta_2 >0,\\
-\frac{1}{(\eta_1-\eta_2)^2
}, & \text{ if } \Im \eta_1 \cdot \Im \eta_2 <0.
\end{cases}
$$
Futhermore, if we assume that $\Im \eta_1, \Im \eta_2>0$ we have 
\begin{multline*}
\lim_{n\to \infty} \Tr H\left(\frac{a_{\lambda_1}}{n^\alpha} \right) H \left(\frac{\tilde a_{\lambda_2}}{n^\alpha} \right)
=\frac{1}{2} \lim_{n\to \infty} \Re H(g_{\lambda_1}) H( \tilde g_{\overline{\lambda}_2})\\
= -\frac{1}{2} \Re \frac{1}{(\eta_1-\bar{\eta_2})^2}.
\end{multline*}
By using $\phi^{(n)}  = \sum_{j=1}^N c_j a_{\lambda_j}$ the statement now follows.
\end{proof}
\subsection{Proof of Theorem \ref{thm:CLTgeneral} in case $f(x)= \sum_{j=1}^N c_j \Im \frac{1}{x-\eta_j}$}
We are now ready to prove a Central Limit Theorem \ref{thm:CLTgeneral} for functions of the form \eqref{eq:defffree}. By then using Theorem \ref{thm:2.1-for-Green} we then obtain a proof of Theorem \ref{thm:CLTgeneral} for functions of that type.
\begin{theorem}\label{thm:freeCLTgreen}
Consider the OPE with  $a_n=1$ and $b_n=0$. Fix $x_0 \in (-2,2)$ and $0<\alpha<1$. Let $N \in \bbN$, $\eta_1,\ldots, \eta_N \in \bbH_+$, $c_1, \ldots,c_N \in \bbR$  and set 
$$f(x)= \sum_{j=1}^n c_j \Im \frac{1}{x-\eta_j}.$$
Then $$X_{f,\alpha,x_0}^{(n)}-\bbE X_{f,\alpha,x_0}^{(n)} \to N(0,\sigma_f^2)$$
in distribution, where $\sigma_f^2$ is given in \eqref{eq:CLTgeneralvariance}. 
\end{theorem}
\begin{proof}
It is enough to prove that $$\mathcal C_m^{(n)}(X_{f,\alpha,x}^{(n)}) \to \begin{cases}
\frac{\sigma_f^2}{2} , & \text{ if }  m =2\\
0, & \text{ if } m >2.
\end{cases}$$
By \eqref{eq:cumulant-spectral-form} we have that
$$\mathcal C_m^{(n)}(X_{f,\alpha,x}^{(n)}) =\mathcal C_m^{(n)}\left(n^{-\alpha} F\left(\vec \lambda,\vec c \right) \right), $$
so by \eqref{eq:HankelThrow} we have 
$$\lim_{n \to \infty} \mathcal C_m^{(n)}\left(X_{f,\alpha,x}^{(n)}\right) = \lim_{n \to \infty} \mathcal C_m^{(n)} \left(T\left(\phi^{(n)}\right)\right).$$
Now, by combining \eqref{eq:borodinokounkov} with Lemmas \ref{lem:smallnormproof} and \ref{lem:variancefree} we see that 
\begin{multline*}
\lim_{n\to \infty} \det\left(I+ P_n({\rm e}^{ t T(\phi^{(n)}))}-I)P_n\right)  {\rm e}^{-t \Tr P_n T (\phi^{(n)})}\\= \exp\left( -\frac{t^2}{4}\sum_{j=1}^n c_i c_j \Re \frac{1}{(\eta_i-\eta_j)^2}\right),
\end{multline*}
uniformly for $t$ in sufficiently small neighborhoods of the origin. 
By \eqref{eq:cumultoepl}, this implies that 
$$\mathcal  C_m^{(n)}(X_{f,\alpha,x}^{(n)}) \to \begin{cases}
-\frac{1}{4}\sum_{j=1}^n c_i c_j \Re \frac{1}{(\eta_i-\eta_j)^2} , & \text{ if }  m =2\\
0, & \text{ if } m >2.
\end{cases}$$
It remains to show that for $m=2$ this is indeed the variance $\sigma_f^2$ as in \eqref{eq:CLTgeneralvariance}. To this end, we note that 
\begin{multline}\label{eq:normphitore}
\iint \left(\frac{f(x)-f(y)}{x-y}\right)^2 {\rm d} x {\rm d} y\\
=\sum_{i,j=1}^N c_i c_j \iint \left(\frac{\Im \frac{1}{x-\eta_i}-\Im \frac{1}{y-\eta_i} }{x-y}\right)\left(\frac{\Im \frac{1}{x-\eta_j}-\Im \frac{1}{y-\eta_j} }{x-y}\right) {\rm d} x {\rm d} y
\end{multline}
Each of the double integrals can be written in the following way
\begin{multline*}
\iint \left(\frac{\Im \frac{1}{x-\eta_i}-\Im \frac{1}{y-\eta_i} }{x-y}\right)\left(\frac{\Im \frac{1}{x-\eta_j}-\Im \frac{1}{y-\eta_j} }{x-y}\right) {\rm d} x {\rm d} y\\
=
\iint \Im \frac{1}{x-\eta_i}  \frac{1}{y-\eta_i} \Im \frac{1}{x-\eta_j} \frac{1}{y-\eta_j}  {\rm d} x {\rm d} y\\
=-\frac{1}{2} \Re \iint  \frac{1}{x-\eta_i}  \frac{1}{y-\eta_i}  \frac{1}{x-\eta_j} \frac{1}{y-\eta_j}  {\rm d} x {\rm d} y
+\frac{1}{2} \Re \iint  \frac{1}{x-\eta_i}  \frac{1}{y-\eta_i}  \frac{1}{x-\overline{\eta_j}} \frac{1}{y-\overline{\eta_j}}  {\rm d} x {\rm d} y\\
=-\frac{1}{2} \Re \left(\int  \frac{1}{x-\eta_i} \frac{1}{x-\eta_j}  {\rm d} x \right)^2
+\frac{1}{2} \Re\left( \int  \frac{1}{x-\eta_i}   \frac{1}{x-\overline{\eta_j}}  {\rm d} x\right)^2
\end{multline*}
The last two integrals can be easily computed by a residue calculus. We recall that both $\Im \eta_1 , \Im \eta_2>0$. Then the first integral vanishes and the second gives (by taking $(2 \pi {\rm i})^2$ into account)
$$\frac{1}{4 \pi^2}\iint \Im \frac{1}{x-\eta_i}  \frac{1}{y-\eta_i} \Im \frac{1}{x-\eta_j} \frac{1}{y-\eta_j}  {\rm d} x {\rm d} y=-\frac{1}{2} \Re\left(\frac{1}{\eta_i-\overline{\eta_j}}\right)^2$$
By inserting this back into \eqref{eq:normphitore} we see that indeed 
$$\mathcal C_m^{(n)}\left( X_{f,\alpha,x_0}^{(n)}\right) \to \frac{1}{2} \sigma_f^2,$$
and the statement is proved.
\end{proof}
\begin{corollary}\label{cor:CLTgeneralspecial}
Theorem \ref{thm:CLTgeneral} holds for $f(x)= \sum_{j=1}^N c_j \Im \frac{1}{x-\eta_j}$ with the same parameters as in Theorem \ref{thm:freeCLTgreen}.  
\end{corollary}
\begin{proof}
Without loss of generality we can assume that $a_n \to 1$ and $b_n \to 0$. Indeed, if $a_n \to a>0$ and $b_n \to b$ then the map $x \mapsto (x-b)/a$ maps the measure  ${\rm d} \mu(x)$ to a measure  ${\rm d} \tilde \mu$ for which the recurrence coefficients converge as $\tilde a_n \to 1$ and $\tilde b_n\to 0$. 

To prove that the linear statistic converges to a normal distribution with given variance it is sufficient to prove that the cumulants converge to the cumulants of that normal distribution.  We already proved that this is true for  $a_n=1$ and $b_n=0$. The general case now follows by Theorem \ref{thm:2.1-for-Green} after showing that conditions $(i)$ and $(ii)$ of that theorem hold for the measure, $\mu_0$, corresponding to the free case $a_n\equiv 1$ and $b_n\equiv 0$. 

Condition $(i)$, namely that $\mu_0$  is  absolutely continuous on $[-2,2]$ with respect to the Lebesgue measure with a Radon-Nykodim derivative that is bounded on compact subsets, holds since ${\rm d}\mu_0(x)=\frac{\sqrt{4-x^2}}{2\pi}{\rm d} x$. Moreover, the orthonormal polynomials with respect to $\mu_0$ are the Chebyshev polynomials of the second kind, which are clearly bounded on any compact subsets of $(-2,2)$. This is condition $(ii)$, and therefore the statement follows.
\end{proof}

%%%%%%%%%%%%%%%%%%%%%%%%%%%%%%%%%%%%%%%%%%%%%

\section{Proofs of Theorems \ref{thm:main-result} and  \ref{thm:CLTgeneral} for $f \in C_c^1(\bbR)$} \label{sec:extension}

In this section we extend the results in Theorem \ref{thm:2.1-for-Green} and Corollary \ref{cor:CLTgeneralspecial} so that they hold for functions $f\in C_c^1(\bbR)$. This allows us to prove Theorems \ref{thm:main-result} and \ref{thm:CLTgeneral}.

\subsection{Estimating the variances}

An important role in the extension argument is played by the following space  of functions. Let  $f:\bbR \to \bbR$ be a function such that 
\begin{enumerate}
\item $\lim_{x\to -\infty} f(x)=0$ 
\item $\|f\|_{\mathcal L_w}:=\sup_{x,y\in \bbR} \sqrt{1+x^2}\sqrt{1+y^2} \left|\frac{f(x)-f(y)}{x-y}\right| <\infty.$
\end{enumerate}
The space of all functions that satisfy these two properties is a normed space with norm $\|f\|_{\mathcal L_w}$ (which is a weighted Lipschitz norm).  This space is denoted by $\mathcal L_w$ and is very useful for extending fluctuation results to allow more general classes of functions in the linear statistic. For this reason it was also used for example in \cite{DJ}.

 Obviously, the functions in $\mathcal L_w$ have the property that 
$$\sqrt{1+x^2}\sqrt{1+y^2} \left|\frac{f(x)-f(y)}{x-y}\right|\leq \|f\|_{\mathcal L_w}$$
for all $x,y \in \bbR$ and by setting $y\to -\infty$ we see that 
$$ |f(x)| \leq \frac{\|f\|_{\mathcal L_w}}{\sqrt{1+x^2}}.$$
Second, in case $f$ is differentiable, we can take the limit $y \to x$ and find 
$$ |f'(x)| \leq \frac{\|f\|_{\mathcal L_w}}{1+x^2}.$$
Hence, we have a bound on the behavior at infinity of both these functions and their derivative (in case it exists). Moreover, we have $\|f\|_\infty \leq \|f\|_{\mathcal L_w}$ and $\|f'\|_{\infty} \leq \|f\|_{\mathcal L_w}.$

A first indication that  $\mathcal L_w$ is a useful space for extending Central Limit Theorems is the fact that the limiting variance in \eqref{eq:CLTgeneralvariance} is continuous with respect to $\|\cdot\|_{\mathcal L_w}$. More precisely,
\begin{equation}
\label{eq:boundonlimitingvariance}
\sigma_f^2=\frac{1}{4 \pi^2}\iint_{\bbR^2} \left(\frac{f(x)-f(y)}{x-y}\right)^2 {\rm d} x {\rm d} y \leq \frac{1}{4} \|f\|_{\mathcal L_w}^2.
\end{equation}
But more is true: the next lemma says that the variance for any OPE for finite $n \in \bbN$ and any $f \in {\mathcal L_w}$ can be estimated in terms of  $\|f\|_{\mathcal L_w}$ and the variance of the linear statistic associated with $g(x)=(x-i)^{-1}$. Note that since $g$ is complex valued, the variance is also that of a compex random variable, i.e.
\beq \no
\Var X^{(n)}_{g}=\int \int _{\bbR^2}\left| g(x)-g(y) \right|^2 K_n(x,y)^2 {\rm d}\mu(x) {\rm d}\mu (y).
\eeq

\begin{proposition}\label{prop:continuityvariance1}
Let $g(x)= (x-i)^{-1}$. There exists a constant $c>0$ such that
\begin{equation}
\Var X^{(n)}_{f,x_0,\alpha} \leq   \|f\|_{\mathcal L_w}^2 \Var X^{(n)}_{g,x_0,\alpha},
\end{equation}
for any $f\in \mathcal L_w$ and $n \in \bbN$. 
\end{proposition}
\begin{proof} For any function $h:\bbR \to \bbR$ the variance of the linear statistic takes the form
$$\Var \sum_{j=1}^n h(x_j) = \frac{1}{2} \iint_{\bbR^2} (h(x)-h(y))^2 K_n(x,y)^2 {\rm d}\mu(x) {\rm d}\mu(y).$$
Hence if $\psi:\bbR \to \bbC$ is another function we can compare the linear statistic by writing 
\begin{multline}\label{eq:boundvarianceh}
\Var \sum_{j=1}^n h(x_j)  \leq \frac{1}{2} \left(\sup_{x,y\in \bbR} \left|\frac{h(x)-h(y)}{\psi(x)-\psi(y)}\right|\right)^2 \iint_{\bbR^2} |\psi(x)-\psi(y)|^2 K_n(x,y)^2 {\rm d}\mu(x) {\rm d}\mu(y).
\end{multline}
We apply this inequality to the special case $h(x)=f(n^\alpha(x-x_0))$ and $\psi(x)=g\left(n^\alpha(x-x_0) \right)=(n^\alpha(x-x_0)-{\rm i})^{-1}$. With this choice we find
\begin{multline}\label{eq:boundofrhwrtg}
\sup_{x,y\in \bbR} \left|\frac{h(x)-h(y)}{\psi(x)-\psi(y)}\right|= \sup_{x,y\in \bbR} \left|\frac{f(n^\alpha(x-x_0))-f(n^\alpha(y-x_0))}{(n^\alpha(x-x_0)-{\rm i})^{-1}-(n^\alpha(x-x_0)-{\rm i})^{-1}}\right|\\
=\sup_{x,y\in \bbR} \left|\frac{f(x)-f(y)}{(x-{\rm i})^{-1}-(y-{\rm i})^{-1}}\right|= \sup_{x,y\in \bbR} \left|\frac{f(x)-f(y)}{x-y} \right||x-{\rm i} ||y-{\rm i}|\\
= \sup_{x,y\in \bbR} \left|\frac{f(x)-f(y)}{x-y} \right|\sqrt{1+x^2}\sqrt{1+y^2}
=\|f\|_{\mathcal L_w}.
\end{multline}
By inserting this back into \eqref{eq:boundvarianceh} the statement follows
\end{proof}

Hence, to obtain a bound for for the variance for any $f\in \mathcal L_w$ it suffices to bound the variance for the single test function $g(x)= (x-i)^{-1}$. The next proposition establishes such a bound under the local reularity conditions on the measure $\mu$ specified in Theorem \ref{thm:main-result}. The generality and independence of scale of this bound make this a result of independent interest. In particular, by \cite[Theorem 2.1]{BD-Nevai} this implies that for $\mu_0$, $x_0$, and $\alpha>0$ as in Proposition \ref{prop:continuityvariance2} below, and any $f \in {\mathcal L_w}$
\beq \no
\mathbb{P} \left( \left|X_{f,\alpha,x_0}-\mathbb{E} X_{f,\alpha,x_0} \right|>\varepsilon \right) \leq 2 {\rm exp} \left(-A \varepsilon \right)
\eeq for some constant $A$.

\begin{proposition}\label{prop:continuityvariance2}
Let $\mu_0$ be a compactly supported Borel measure on $\bbR$ with finite moments and with Jacobi coefficients $\{a_n, b_n \}_{n=1}^\infty$. Let $0<\alpha<1$ and $x_0 \in S(\mu_0)$ such that there exists a closed interval $I$ around $x_0$ for which \\
(i) the measure $\mu_0$ is absolutely continuous on $I$ with respect to the Lebesgue measure and its Radon-Nikodym derivative is bounded there.   \\
 (ii) the orthonormal polynomials $p_n$ for $\mu_0$ are uniformly bounded on $I$.\\
Then  there exists a constant $c>0$ such that for any $n \in \bbN$
\beq \label{eq:boundedvariance}
 \Var X_{g,\alpha,x_0}^{(n)}\leq c.
\eeq
In particular, for any  $f\in \mathcal L_w$ and any $n \in \bbN$
 \begin{equation}\label{eq:continuityvariance}
\Var X_{f,\alpha,x_0}^{(n)}\leq c  \|f\|_{\mathcal L_w}^2.
\end{equation}

Moreover, if $\mu$ is a second measure Borel measure with finite moments for which $\tilde a_n-a_n = \mathcal O(n^{-\beta})$ and $\tilde b_n-b_n= \mathcal O(n^{-\beta})$ as $n\to \infty$, for some $\beta>\alpha$. Then  \eqref{eq:continuityvariance} also holds for the OPE corresponding to $\mu$. 
\end{proposition}
\begin{proof}
In view of Proposition \ref{prop:continuityvariance1} we only need to prove that 
\begin{equation}
\label{eq:variancegstep0}
\Var X_{g,\alpha,x_0}^{(n)}= \frac12 \iint_{\bbR^2} |g(x)-g(y)|^2 K_n(x,y)^2 {\rm d}\mu_0(x) {\rm d}\mu_0(y) <c,
\end{equation}
for some constant $c>0$. To this end, we first note that 
\begin{multline*}
|g(x)-g(y)|^2 =\frac{1}{n^{2 \alpha}} \frac{(x-y)^2}{|x-x_0-{\rm i} /n^\alpha|^2 |y-x_0-{\rm i} /n^\alpha|^2}
\\
=(x-y)^2\Im  \frac{1}{x-x_0-{\rm i} /n^\alpha} \Im \frac{1}{y-x_0-{\rm i} /n^\alpha}
\end{multline*}
and hence, by the Christoffel-Darboux formula for $K_n(x,y)$, we have
\begin{multline*}
\iint_{\bbR^2} |g(x)-g(y)|^2 K_n(x,y)^2 {\rm d}\mu_0(x) {\rm d}\mu_0(y)\\
=\iint_{\bbR^2} \Im  \frac{1}{x-x_0-{\rm i} /n^\alpha} \Im \frac{1}{y-x_0-{\rm i} /n^\alpha} \\ \times (p_n(x)p_{n-1}(y)-p_n(y) p_{n-1}(x))^2 {\rm d}\mu_0(x) {\rm d}\mu_0(y).
\end{multline*}
By expanding the brackets we obtain three terms, each of which is a product of two single integrals
\begin{multline} \label{eq:varinacegstep1}
\iint_{\bbR^2} |g(x)-g(y)|^2 K_n(x,y)^2 {\rm d}\mu_0(x) {\rm d}\mu_0(y)\\
=a_n^2 \int_{\bbR} \Im  \frac{p_n(x)^2}{x-x_0-{\rm i} /n^\alpha}  {\rm d}\mu_0(x) \int_{\bbR} \Im \frac{ p_{n-1}(y)^2}{y-x_0-{\rm i} /n^\alpha}  {\rm d}\mu_0(y)\\
-2 a_n^2 \int_{\bbR} \Im  \frac{p_n(x) p_{n-1}(x)}{x-x_0-{\rm i} /n^\alpha}   {\rm d}\mu_0(x) \int_{\bbR} \Im \frac{p_{n-1}(y) p_n(y)}{y-x_0-{\rm i} /n\alpha}   {\rm d}\mu_0(y)\\
 +a_n^2 \int_{\bbR} \Im  \frac{p_{n-1}(x)^2}{x-x_0-{\rm i} /n^\alpha}  {\rm d}\mu_0(x) \int_{\bbR} \Im \frac{ p_{n}(y)^2 }{y-x_0-{\rm i} /n^\alpha} {\rm d}\mu_0(y).
\end{multline}
Now using he fact that the resolvent has the form 
$$\left(G_0(\lambda)\right)_{jk} = \int \frac{p_j(x)p_k(x) }{x-\lambda} {\rm d} \mu_0,$$
we thus find that
\begin{multline} \label{eq:varinacegintermsofresolvent}
\iint_{\bbR^2} |g(x)-g(y)|^2 K_n(x,y)^2 {\rm d}\mu_0(x) {\rm d}\mu_0(y)\\
=2a_n^2 \left(\Im \left(G_0(\lambda_n)\right)_{n,n}(\Im \left(G_0(\lambda_n)\right)_{n-1,n-1}-(\Im \left(G_0(\lambda_n)\right)_{n,n-1})^2\right),
\end{multline}
with $\lambda_n= x_0+ {\rm i }/n^\alpha$. Hence the variance is a determinant of a $2\times 2$ matrix where the entries are particular entries of the resolvent (This is an interesting observation in its own right). 

The first part of the proposition now follows after observing that the conditions for $\mu_0$ imply that $\Im\left(G_0(\lambda_n)\right)_{n-1,n-1}$,$\Im\left(G_0(\lambda_n)\right)_{n,n-1}$ and $\Im\left(G_0(\lambda_n)\right)_{n,n}$ are bounded, which we already proved in \eqref{eq:boundednessgfunction} and the discussion directly below. The second part then follows from Proposition \ref{prop:comparinggreen}.
 \end{proof}

The question now rises as to which functions we can approximate in $\mathcal L_w$-norm by functions of the form $f(x)= \sum_{j=1}^N c_j \Im \frac{1}{x-\eta_j}$. While the optimal space of functions may be larger than indicated by the following lemma, the resulting space, $C_c^1(\bbR)$, suffices for our purposes here.

\begin{lemma}\label{lem:fromlipschitztoC1c}
Let $ f\in C_c^1(\bbR)$ $($the space of continuously differentiable function with compact support$)$. For any $\eps>0$, there exists $N\in \bbN, c_1,\ldots,c_N$ and $\eta_1,\ldots,\eta_N\in \bbH_+= \{\Im \eta>0\}$ such that 
$$\left \|f(x)-\sum_{j=1}^N c_j \Im \frac{1}{x-\eta_j} \right \|_{\mathcal L_w} <\eps.$$
\end{lemma}
\begin{proof}
For $\delta>0$ define 
$$f_\delta(x)= \frac{1}{\pi} \int_{\bbR} f(s) \Im \frac{1}{s-x-{\rm i} \delta}{\rm d} s.$$
It is a standard approximation result that $\|f-f_{\delta}\|_\infty \downarrow  0$ as $\delta \downarrow 0$. In fact, \cite[Lemma 4.5]{DJ} shows that this approximation also holds in $\mathcal L_w$, i.e.\ $\|f-f_\delta \|_{\mathcal L_w} \downarrow 0$ as $\delta \downarrow 0$. Hence it is sufficient to prove the statement for $f_\delta$ with $0<\delta <1$.

Note that for any $\phi$ we have 
\begin{multline*}
\|f_\delta-\phi_\delta\|_{\mathcal L_w} \leq  \int_\bbR  |f(s)-\phi(s)| \\\times \sup_{x,y\in \bbR} \left| \frac{\sqrt{1+x^2}\sqrt{1+y^2}}{x-y} \Im \left( \frac{1}{x-s-{\rm i} \delta}-\frac{1}{y-s-{\rm i} \delta}\right)\right|{\rm d} s\\
=  \int_\bbR  |f(s)-\phi(s)| \sup_{x,y\in \bbR} \left|\Im  \frac{\sqrt{1+x^2}\sqrt{1+y^2}}{(x-s-{\rm i} \delta)(y-s-{\rm i} \delta)}\right|{\rm d} s.
\end{multline*}
We simplify the weight in the integral using 
\begin{equation} \label{eq:beurlingweight} 
4(1+a^2)(1+b^2)\geq 2(1+(a+b)^2), \qquad \text{ for } a,b \in \bbR.
\end{equation}
After some algebra (by choosing $a= (s-x)/\delta$ and $b=s/\delta$ and using the fact that $0<\delta<1$), we see that  
$$\frac{\sqrt{1+x^2}}{\sqrt{(x-s)^2+ \delta^2}} \leq \delta^{-1} {\sqrt{2 (1+s^2/\delta^2)}},$$
and therefore 
\begin{equation}\label{eq:contLwinLweigth}
\|f_\delta-\phi_\delta\|_{\mathcal L_w} 
\leq \frac{2}{\delta^2} \int |f(s)-\phi(s) |  (1+s^2/\delta^2)  {\rm d} s.
\end{equation}
 Moreover, if \begin{equation}
\label{eq:defphisumlin}\phi(x)=\sum_{j=1}^{N'} c_j'  \Im \frac{1}{x-\eta_j'},\end{equation} then it is a simple calculation (by either a residue calculation or by computing Fourier transforms) that 
\begin{equation}\label{eq:defphisumlin2}
\phi_{\delta}(x)=\sum_{j=1}^N c_j' \Im \frac{1}{x-\eta_j-{\rm i} \delta},
\end{equation}
for some $c_j\in \bbR$ and $\eta_j \in \bbH_+$. Hence if we can approximate $f$ with a function $\phi$ of the  form \eqref{eq:defphisumlin} in $\mathbb L_1(\bbR, (1+s^2/\delta^2) {\rm d}s) $, then $\phi_\delta$ approximates $f_\delta$ in $\mathcal L_w$ and $\phi_\delta$ is of the desired form.

From \eqref{eq:beurlingweight} 
we see that  the weight function $w(s)=2(1+s^2/\delta^2)$ is a Beurling weight (i.e., $w(s_1)w(s_2) \geq w(s_1+s_2)$) and
hence by Beurling's extension (see for example \cite[Theorem V.4.1]{Korevaar}) of Wiener's Tauberian Theorem we see that the linear span of  $\{s\mapsto \psi(s-t)\}_{t \in \bbR} $ is dense in $\mathbb L_1(\bbR, w(s) {\rm d} s)$ if and only if the Fourier transform $\hat \psi$ never vanishes.  We choose this function to be 
$$\psi(s)=\Im\left( \frac{1}{s-{\rm i}}-\frac{1}{s-2 {\rm i}}\right)= \frac{1}{1+s^2}-\frac{1}{4+s^2}.
$$
Note that by the asymptotic behavior at $\pm \infty$, we have that $\psi \in L_1(\bbR, w(s) {\rm d} s)$. Moreover, one easily checks that the Fourier transform satisfies $\hat \psi>0$ and hence, for any $\eps >0$
there exists $N'\in \bbN$,  $c_1', \ldots,c_N'\in \bbR$ and $\tau_1,\ldots,\tau_N\in \bbR$ such that 
the function defined by 
$$\phi(x)=\sum_{j=1}^{N'} c_j'  \psi(s-\tau_j),$$
satisfies 
$$\int |f(s)- \phi(s)|   (1+s^2/\delta^2)  {\rm d} s \leq \delta^2 \eps/2.$$
Together with \eqref{eq:contLwinLweigth} and the fact that $\phi_\delta$ is of the desired form  \eqref{eq:defphisumlin2}, this proves the statement. 
\end{proof}

\subsection{Proofs of Theorems \ref{thm:main-result} and \ref{thm:CLTgeneral}}

We are now ready to prove Theorem \ref{thm:main-result} and \ref{thm:CLTgeneral}. 

\begin{proof}[Proof of Theorem \ref{thm:main-result}]
Theorem \ref{thm:main-result} for functions of the form $f(x)= \sum_{j=1}^N c_j (x-\eta_j)^{-1}$ is Theorem \ref{thm:2.1-for-Green}. For general $f\in C_c^1(\bbR)$ we use
\begin{multline*}
\left|\bbE \left(X^{(n)}_{f,\alpha,x_0}-\bbE  X^{(n)}_{f,\alpha,x_0} \right)^m -\bbE_0 \left(X^{(n)}_{f,\alpha,x_0}-\bbE_0 X^{(n)}_{f,\alpha,x_0}\right)^m\right| \\
\leq 
\left|\bbE \left(X^{(n)}_{f,\alpha,x_0} -\bbE  X^{(n)}_{f,\alpha,x_0} \right)^m -\bbE \left(X^{(n)}_{g,\alpha,x_0}-\bbE X^{(n)}_{g,\alpha,x_0}\right)^m\right| \\
+\left|\bbE \left(X^{(n)}_{g,\alpha,x_0}-\bbE  X^{(n)}_{g,\alpha,x_0} \right)^m -\bbE_0 \left(X^{(n)}_{g,\alpha,x_0}-\bbE_0 X^{(n)}_{g,\alpha,x_0}\right)^m\right| \\
\left|\bbE_0 \left(X^{(n)}_{f,\alpha,x_0} -\bbE_0  X^{(n)}_{f,\alpha,x_0} \right)^m -\bbE_0 \left(X^{(n)}_{g,\alpha,x_0}-\bbE_0 X^{(n)}_{g,\alpha,x_0}\right)^m\right|,
\end{multline*}
where $g(x)= \sum_{j=1}^N c_j (x-\eta_j)^{-1}$.
By Lemma \ref{lem:continuitymoments} together with  Proposition \ref{prop:continuityvariance2} and Lemma \ref{lem:fromlipschitztoC1c}, we see that we can choose $N, c_j$ and $\eta_j$ such that both the first and the third term can be made arbitarily small. By Theorem \ref{thm:2.1-for-Green} the middle term on the right-hand side then tends to zero as $n \to \infty$. This proves the statement. 
\end{proof}

\begin{proof}[Proof of Theorem \ref{thm:CLTgeneral}]
We already proved Corollary \ref{cor:CLTgeneralspecial} which is the statement in case $f$ is of the form $f(x)= \sum_{j=1}^N c_j (x-\eta_j)^{-1}$. For general $f\in C^1_c(\bbR)$, we use Lemma \ref{lem:continuitymoments} and Proposition \ref{prop:continuityvariance2} to deduce  that  for any $g(x)= \sum_{j=1}^N c_j (x-\eta_j)^{-1}$ we have, for $m \geq 3$,  
\begin{multline*}
\limsup_{n\to \infty} \left|\mathcal C_m^{(n)} (X^{(n)}_{f,\alpha,x_0})\right|\\ \leq \limsup_{n\to \infty} \left|\mathcal C_m^{(n)} (X^{(n)}_{g,\alpha,x_0}) \right| + \limsup_{n\to \infty}  \left |\mathcal C_m^{(n)} (X^{(n)}_{f,\alpha,x_0}))-\mathcal C_m^{(n)} (X^{(n)}_{g,\alpha,x_0}) \right|
\\\leq 
 D_1\|f-g\|_{\mathcal L_w} \exp D_2 \left(\|f|_{\mathcal L_w}+ \|f-g\|_{\mathcal L_w}\right).
\end{multline*}
for some constants $D_1,D_2>0$.   By Lemma \ref{lem:fromlipschitztoC1c} the latter can be made arbitrarily small and hence $$\lim_{n\to \infty} \mathcal C_m^{(n)} (X^{(n)}_{f,\alpha,x_0})=0,$$
for $m \geq 3$.

For $m=2$ we use  in addition \eqref{eq:boundonlimitingvariance} and find
\begin{multline*}
\limsup_{n\to \infty} \left|\mathcal C_2^{(n)} (X^{(n)}_{f,\alpha,x_0}))-\frac{1}{2}\sigma^2_f \right|= \limsup_{n\to \infty} \left|\mathcal C_2^{(n)} (X^{(n)}_{g,\alpha,x_0})-\frac12 \sigma_g^2 \right|\\
+\left|\frac12 \sigma_f^2-\frac{1}{2}\sigma_g^2\right| + \limsup_{n\to \infty}  \left |\mathcal C_2^{(n)} (X^{(n)}_{f,\alpha,x_0})-\mathcal C_2^{(n)} (X^{(n)}_{g,\alpha,x_0}) \right|
\\ \leq \tilde D_1 \left \|f-g \right\|_{\mathcal L_w} \exp \tilde  D_2 \left(\|f\|_{\mathcal L_w}+ \|f-g\|_{\mathcal L_w}\right).
\end{multline*}
for some constants $\tilde D_1,\tilde D_2>0$. 
Again, by Lemma \ref{lem:fromlipschitztoC1c} the left-hand side can be made arbitrarily small and so $$\lim_{n\to \infty} \mathcal C_2^{(n)} (X^{(n)}_{f,\alpha,x_0})=\frac{1}{2}\sigma_f^2.$$
This proves the statement.
\end{proof}

\end{document}